\documentclass[12pt]{article}
\usepackage{amssymb,amsmath,epsfig,amscd,eucal,psfrag,amsthm}
\usepackage{graphicx,hyperref}

\newtheorem{theorem}{Theorem}[section]

\newtheorem{lemma}[theorem]{Lemma}
\newtheorem{proposition}[theorem]{Proposition}

\newtheorem{corollary}[theorem]{Corollary}

\newtheorem{letterthm}{Theorem}

\theoremstyle{definition}
\newtheorem{definition}[theorem]{Definition}

\theoremstyle{remark}
\newtheorem{remark}[theorem]{Remark}

\newcommand{\G}{\mathcal{G}}

\newcommand{\Q}{\mathbb{Q}}
\newcommand{\R}{\mathbb{R}}
\newcommand{\Z}{\mathbb{Z}}

\newcommand{\wh}{\widehat}

\newcommand{\tr}{\mathrm{tr}}

\newcommand{\curlyH}{\mathcal{H}}

\newcommand{\dcup} {\mathinner{\cup \mkern -8.7mu \rlap{\raise 0.6ex\hbox{.}}\mkern 8.7mu}}

\input xy
\xyoption{all}

\title{Distinguishing geometries using finite quotients}
\author{Henry Wilton and Pavel Zalesskii}

\begin{document}

\maketitle

\begin{abstract}
We prove that the profinite completion of the fundamental group of a compact 3-manifold $M$ satisfies a Tits alternative: if a closed subgroup $H$ does not contain a free pro-$p$ subgroup for any $p$, then $H$ is virtually soluble, and furthermore of a very particular form.  In particular, the profinite completion of the fundamental group of a closed, hyperbolic 3-manifold does not contain a subgroup isomorphic to $\widehat{\Z}^2$.   This gives a profinite characterization of hyperbolicity among irreducible 3-manifolds.  We also characterize Seifert fibred 3-manifolds as precisely those for which the profinite completion of the fundamental group has a non-trivial procyclic normal subgroup.  Our techniques also apply to hyperbolic, virtually special groups, in the sense of Haglund and Wise.  Finally, we prove that every finitely generated pro-$p$ subgroup of the profinite completion of a torsion-free, hyperbolic, virtually special group is free pro-$p$.
\end{abstract}

In a heuristic commonly used to distinguish 3-manifolds $M$ and $N$, one computes all covers $M_1,\ldots,M_m$ of $M$ and $N_1,\ldots, N_n$ of $N$ up to some small degree, and then compares the resulting finite lists $(M_i)$ and $(N_j)$. If they can be distinguished, then one has a proof that $M$ and $N$ were not homeomorphic.

It is natural to ask whether this method always works.  A more precise question was formulated by Calegari--Freedman--Walker in \cite{calegari_positivity_2010}, who asked whether the fundamental group of a 3-manifold is determined by its finite quotients, or equivalently, by its profinite completion.  (A standard argument shows that two finitely generated groups have the same set of finite quotients if and only if their profinite completions are isomorphic.) Bridson, Conder and Reid have answered the corresponding question for Fuchsian groups positively \cite{bridson_determining_2014}, while  Long and Reid have given a positive answer to a related question \cite{long_grothendieck_2011}.  We refer the reader to Section 8 of \cite{reid_profinite_2013} for a discussion of this and related problems.

Funar \cite{funar_torus_2013} used work of Stebe \cite{stebe_conjugacy_1972b} to exhibit Sol-manifolds that answer the Calegari--Freedman--Walker question in the negative, and more recently Hempel has exhibited Seifert fibred examples \cite{hempel_some_2014}.  Nevertheless, because of the effectiveness of the above heuristic, it remains natural to ask how much information about a given 3-manifold is contained in the profinite completion of its fundamental group.

Our first theorem shows that the hyperbolicity of a closed 3-manifold $M$ is determined by this profinite completion.  Recall that the Sphere theorem and the Kneser--Milnor decomposition imply that a closed, orientable 3-manifold $M$ is aspherical if and only if $M$ is irreducible and its fundamental group is infinite.

\begin{letterthm}\label{thm: Profinite hyperbolization}
Let $M$ be a closed, orientable, aspherical 3-manifold. Then $M$ is hyperbolic if and only if the profinite completion $\widehat{\pi_1M}$ does not contain a subgroup isomorphic to $\wh{\Z}^2$.
\end{letterthm}

This can be thought of as a profinite analogue of the Hyperbolization theorem, which asserts that $M$ is hyperbolic if and only if $\pi_1M$ does not contain a subgroup isomorphic to $\mathbb{Z}^2$.  This is one case of the Geometrization theorem, proved by Perelman \cite{perelman_entropy_2002,perelman_ricci_2003,perelman_finite_2003} (see also \cite{kleiner_notes_2008,morgan_ricci_2007,morgan_geometrization_2014}). We also have the following profinite characterization of Seifert fibred 3-manifolds.

\begin{letterthm}\label{thm: Profinite Seifert conjecture}
Let $M$ be a closed, orientable, aspherical 3-manifold. Then $M$ is Seifert fibred if and only if the profinite completion $\wh{\pi_1M}$ has a non-trivial procyclic normal subgroup.
\end{letterthm}

The Seifert conjecture, which follows from the Geometrization theorem but was proved earlier by Casson--Jungreis \cite{casson_convergence_1994} and Gabai \cite{gabai_convergence_1992}, building on work of Mess \cite{mess_seifert_1987} and Tukia \cite{tukia_homeomorphic_1988}, asserts that if $\pi_1M$ is infinite then $M$ is Seifert fibred if and only if $\pi_1M$ has an infinite cyclic normal subgroup.

In fact, the Geometrization theorem is equivalent to the
Hyperbolization theorem and the Seifert conjecture, together with
the Elliptization theorem, which asserts that $M$ is spherical if
and only if $\pi_1M$ is finite \cite{scott_geometries_1983}. Since
3-manifold groups are residually finite
\cite{hempel_residual_1987}, $\pi_1M$ is finite if and only if its
profinite completion is, so there is no distinct profinite
analogue of the Elliptization theorem.  Theorems \ref{thm:
Profinite hyperbolization} and \ref{thm: Profinite Seifert
conjecture} can therefore be thought of as providing a complete
profinite analogue of the Geometrization theorem (although one
should note, of course, that our proofs rely essentially on
Geometrization). In Theorem \ref{thm: Geometry}, we proceed to
show that the profinite completion of the fundamental group
detects the geometry of a closed, orientable, irreducible
3-manifold.

Alternatively, Theorem \ref{thm: Profinite hyperbolization} can be thought of as a classification result for the abelian subgroups of profinite completions of fundamental groups of hyperbolic manifolds.  In fact, we prove a much more general `Tits alternative' for profinite completions of 3-manifold groups.  We use the notation $\Z_\pi$ to denote $\prod_{p\in \pi} \Z_p$. 

\begin{letterthm}\label{thm: 3M Tits alternative}
If $M$ is any compact 3-manifold and $H$ is a closed subgroup of $\wh{\pi_1M}$ that does not contain a free non-abelian pro-$p$ subgroup for any prime $p$ then $H$ is on the following list:
\begin{enumerate}
\item $H$ is conjugate to the completion of a virtually soluble
subgroup of $\pi_1M$; or \item $H$ is isomorphic to
$((\Z_\sigma\times \Z_\pi)\rtimes C)$, where $\pi$ and $\sigma$
are (possibly empty) sets of primes with $\pi\cap
\sigma=\varnothing$ and $C$ is procyclic (possibly
finite).
\end{enumerate}
\end{letterthm}

Note that virtually soluble 3-manifold groups are classified, so the first possibility is well understood \cite[Theorem 1.11.1]{aschenbrenner_manifold_2015}.  For the second possibility, there are in fact additional constraints on the structure of $H$, depending on the geometry of $M$.  We refer the reader to Theorem \ref{profinite} and Proposition \ref{prop: SF Tits alternative} for details.

The proof of Theorem \ref{thm: Profinite hyperbolization} uses the dramatic recent developments of Agol \cite{agol_virtual_2013}, Kahn--Markovic \cite{kahn_immersing_2012} and Wise \cite{wise_structure_2012}, whose work implies that the fundamental groups of closed hyperbolic 3-manifolds are also fundamental groups of compact, virtually special cube complexes (see \cite{aschenbrenner_manifold_2015} for a summary of these developments).  Indeed, we prove the following theorem.

\begin{letterthm}\label{thm: Hyp special}
Let $G$ be hyperbolic, virtually special group. If $H$ is a closed
subgroup of $\wh{G}$ that does not contain a free non-abelian
pro-$p$ subgroup for any $p$ then $H$ is virtually isomorphic to
$\Z_\pi\rtimes\Z_\rho$, where $\pi$ and $\rho$ are disjoint
(possibly empty) sets of primes. If $G$ is torsion free, then
`virtually' can be omitted.
\end{letterthm}

As well as hyperbolic 3-manifold groups, many other classes of hyperbolic groups are now known to be virtually special: word-hyperbolic Coxeter groups \cite{haglund_coxeter_2010}; $C'(1/6)$ small-cancellation groups \cite{wise_cubulating_2004,agol_virtual_2013}); random groups at density less than 1/6 \cite{ollivier_cubulating_2011,agol_virtual_2013}); and one-relator groups with torsion \cite{wise_structure_2012}.  Most of these examples are torsion-free, with the exception of Coxeter groups and one-relator groups. In the one-relator case,  we can improve our classification.

\begin{letterthm}\label{thm: 1-relator}
Let $G$ be a one-relator group with torsion and $H$ a closed
subgroup of $\wh{G}$ that does not contain a free non-abelian
pro-$p$ subgroup for any $p$. Then $H$ is virtually soluble and
has one of the following forms:
\begin{enumerate} \item
$H\cong\Z_\pi\rtimes\Z_\rho$ where $\pi$ and $\rho$ are disjoint
sets of primes; \item $H$ is a profinite dihedral group
$\Z_\pi\rtimes C_2$; \item $H$ is a profinite Frobenius group
$\Z_\pi\rtimes C_n$, i.e.\ the order of every prime divisor $p$ of
$n$ divides $q-1$ for some $q\in\pi$ and the centralizers of
nonidentity elements of $C_n$ coincide with $C_n$.
\end{enumerate}
\end{letterthm}

We then turn to analyse the pro-$p$ subgroups of $\widehat{G}$, whose structure turns out to be amazingly restricted: they are all free whenever $G$ is torsion-free, hyperbolic and virtually special (such as the fundamental group of a hyperbolic 3-manifold).

\begin{letterthm}\label{thm: pro-p virtually special}
Let $G$ be a torsion-free, hyperbolic, virtually special group (such as the fundamental group of a compact hyperbolic 3-manifold).  Any finitely generated pro-$p$ subgroup $H$ of $\widehat G$ is free pro-$p$.
\end{letterthm}

Also, if $G$ has torsion, then a finitely generated pro-$p$ subgroup $H$ of $\widehat G$ is virtually free pro-$p$ and therefore, by the main result of \cite{herfort_virtually_2013}, admits a a pro-$p$ analog of Karras--Pietrowsky--Solitar's description, i.e.\  as the pro-$p$ fundamental group of a finite graph of finite $p$-groups
(see Corollary \ref{cor:pro-p virtually special}).

\medskip

We conclude this introduction by outlining the structures of the proofs in this paper. Since Theorem \ref{thm: Profinite hyperbolization} follows immediately from Theorem \ref{thm: Hyp special}, we will first illustrate the proof of Theorem \ref{thm: Hyp special} by sketching the proof that the profinite completion of a hyperbolic virtually special group cannot contain a subgroup isomorphic to $\wh{\Z}^2$. 

Any hyperbolic virtually special group $G$ has a subgroup $G_0$ of finite index that admits a malnormal quasiconvex hierarchy (see Definition \ref{defn: MQH}).  There is a well known geometric proof that hyperbolic groups cannot contain subgroups isomorphic to $\mathbb{Z}^2$,  but one can give an alternative proof of this fact for hyperbolic virtually special groups using the malnormal hierarchy, as follows: since the edge groups in the hierarchy are malnormal, the corresponding action of $G_0$ on the Bass--Serre tree is 1-acylindrical (see Definition \ref{def: Acylindrical}); but $\Z^2$ does not admit an acylindrical action on a tree, so any $\Z^2$ subgroup of $G_0$ is conjugate into a lower level of the hierarchy, and we conclude by induction.  Note that, in this proof, only the malnormality of the hierarchy was used; this makes it suitable for translation to the profinite setting, where malnormality makes sense.

The proof of Theorem \ref{thm: Hyp special} is a profinite analogue of this argument.  Classical Bass--Serre theory is replaced by the theory of groups acting on profinite trees (see Sections \ref{sec: Profinite trees} and \ref{sec: Profinite gogs}).  As in the discrete setting, the key fact needed is that the edge stabilizers should be malnormal.    This is provided by the following theorem, which is of independent interest.  (See Theorem \ref{thm: Malnormal closure} for a more comprehensive statement.)

\begin{letterthm}\label{thm: Profinite malnormality intro}
Let $G$ be a hyperbolic virtually special group and let $H$ be a quasiconvex, almost malnormal subgroup.  Then $\wh{H}$ is also almost malnormal in $\wh{G}$.
\end{letterthm}

Theorem \ref{thm: Profinite malnormality intro} is proved using detailed properties of the \emph{canonical completion} for immersions of special cube complexes, developed by Haglund and Wise in \cite{haglund_special_2008} and \cite{haglund_combination_2012} (see Section \ref{sec: Canonical completion} for details).  This completes our sketch of the proof of Theorem \ref{thm: Hyp special} (and hence Theorem \ref{thm: Profinite hyperbolization}).

The above argument implies in particular that, after passing to a finite-index subgroup, the profinite completion of the fundamental group of a closed hyperbolic 3-manifold admits an acylindrical action on a profinite tree (see Definition \ref{def: Profinite acylindrical}).  In fact, the same holds true for every closed 3-manifold $M$ which is neither Seifert fibred nor a Sol-manifold: if $M$ is reducible then the profinite tree is induced by the Kneser--Milnor decomposition, and if $M$ is irreducible and not geometric then the profinite tree is induced by the JSJ decomposition.   The bulk of the proof of Theorem \ref{thm: Profinite Seifert conjecture} consists of these observations, together with a result (Proposition \ref{prop:normal}) which implies that a profinite group acting acylindrically on a profinite tree cannot have a non-trivial procyclic normal subgroup.  Again, this is the profinite analogue of a well known lemma from classical Bass--Serre theory. A separate argument is needed to show that profinite completions of Sol-manifolds do not admit non-trivial procyclic normal subgroups.  With Theorem \ref{thm: Profinite Seifert conjecture} in hand, we go on to show that the profinite completion of the fundamental group of a geometric 3-manifold determines its geometry (see Theorem \ref{thm: Geometry}).

As in the case of Theorem \ref{thm: Hyp special}, the proof of Theorem \ref{thm: 3M Tits alternative} is by induction on a suitable hierarchy.  By applying first the (profinite) Kneser--Milnor decomposition and then the (profinite) JSJ decomposition, the theorem is reduced to the cases of Seifert-fibred and hyperbolic manifolds, possibly with cusps.  The main difficulty at this point is provided by hyperbolic manifolds with cusps; as in the closed case, one needs to know that (after passing to a finite-sheeted cover) there is a suitable hierarchy in which the corresponding actions on profinite trees are profinitely acylindrical.  We resolve this difficulty with a combinatorial Dehn filling argument (see Subsection \ref{subsec: Cusp}), which reduces the problem to the setting of hyperbolic virtually special groups, where Theorem \ref{thm: Profinite malnormality intro} applies.

The proof of Theorem \ref{thm: 1-relator} amounts to a careful analysis of the hierarchy used by Wise to show that one-relator groups with torsion are virtually special.  Finally, the main additional ingredient of Theorem \ref{thm: pro-p virtually special} is Theorem \ref{pro-p acting on profinite}, of interest in its own right, which asserts that a pro-$p$ group acting 1-acylindrically on a profinite tree actually splits as a free (pro-$p$) product.

\subsection*{Acknowledgements}

The authors would like to thank Martin Bridson, Daniel Groves, Jason Manning, and especially Alan Reid, who suggested Theorems \ref{thm: Profinite hyperbolization} and \ref{thm: Profinite Seifert conjecture}.  {The authors would also like to thank Gareth Wilkes for pointing out the error in the proof of \cite[Lemma 4.7]{hamilton_separability_2013}, and Emily Hamilton for assistance with the corrected proof.  (See Lemma \ref{lem: Peripheral profinite malnormality} below.)}  The first author is supported by the EPSRC and the second author by CNPq and CAPES.

\section{Hierarchies for virtually special groups}

In \cite{haglund_special_2008}, Haglund and Wise defined a non-positively curved cube complex to be \emph{special} if its hyperplanes do not exhibit a certain finite list of pathologies (they are not one-sided, and do not self-intersect, directly self-osculate, or inter-osculate). The reader is referred to that paper for details, but we will not need them here.

We will call a group $G$ \emph{virtually special} if some finite-index subgroup is the fundamental group of a compact, special cube complex.  (Note that, in some other contexts, the compactness hypothesis is omitted.  This leads to a different class of groups.)

Since the hyperplanes in a special cube complex are embedded and two-sided, one can cut along them to obtain a hierarchical decomposition of the fundamental group.  This can be thought of as a more geometric version of the hierarchy admitted by a Haken 3-manifold.  Indeed, in the word-hyperbolic case, results of Haglund and Wise characterize virtually special groups as precisely those that admit a sufficiently well behaved hierarchy.

In order to state their results, we first need to describe the sorts of hierarchies that we are interested in.

\begin{definition}
The class of word-hyperbolic groups with a \emph{quasiconvex hierarchy} is the smallest class of groups, closed under isomorphism, that contains the trivial group, and such that, if
\begin{enumerate}
\item $G=A*_CB$ and $A$, $B$ each have a quasiconvex hierarchy, or
\item $G=A*_C$ and $A$ has a quasiconvex hierarchy,
\end{enumerate}
and $C$ is quasiconvex in $G$, then $G$ also has a quasiconvex hierarchy.
\end{definition}

The \emph{quasiconvex} subgroups referred to above are the `geometrically well behaved' subgroups of a word-hyperbolic group.  The following, algebraic, notion of `good behaviour' will also concern us.

\begin{definition}
A subgroup $H\subseteq G$ is called \emph{malnormal} if
\[
H^\gamma\cap H=1
\]
whenever $\gamma\notin H$. In groups $G$ with torsion, it is often more appropriate to consider \emph{almost malnormal} subgroups: a subgroup $H$ is almost malnormal if $H^\gamma\cap H$ is finite whenever $\gamma\notin H$; in particular, if $G$ is torsion-free then almost malnormal subgroups are malnormal.
\end{definition}

It is frequently also useful to generalize this definition to families of subgroups: a family $\mathcal{H}=\{H,\ldots,H_n\}$ of subgroups of a group $G$ is called \emph{malnormal} if
\[
H_i\cap H_j^\gamma\neq 1\Rightarrow i=j\mathrm{~and~}\gamma\in H_i
\]
for any $\gamma\in G$ and any indices $i,j$.   As above, one may also relax the definition of a malnormal family of subgroups to obtain a notion of an \emph{almost malnormal} family of subgroups.

We can use this definition to define a more restrictive notion of hierarchy.

\begin{definition}\label{defn: MQH}
The class of word-hyperbolic groups with a \emph{malnormal quasiconvex hierarchy} is the smallest class of groups, closed under isomorphism, that contains the trivial group, and such that, if
\begin{enumerate}
\item $G=A*_CB$ and $A$, $B$ each have a malnormal  quasiconvex hierarchy, or
\item $G=A*_C$ and $A$ has a malnormal quasiconvex hierarchy,
\end{enumerate}
and $C$ is malnormal and quasiconvex in $G$, then $G$ also has a malnormal quasiconvex hierarchy.
\end{definition}

With these definitions in hand, we can summarize some deep results of Haglund, Hsu and Wise \cite{haglund_combination_2012,hsu_cubulating_2015,wise_structure_2012} in the following theorem.

\begin{theorem}[Haglund, Hsu and Wise]\label{thm: Hierarchies}
Let $G$ be a word-hyperbolic group.  The following are equivalent:
\begin{enumerate}
\item $G$ is virtually special;
\item $G$ has a subgroup $G_0$ of finite index with a malnormal quasiconvex hierarchy;
\item $G$ has a subgroup $G_1$ of finite index with a quasiconvex hierarchy.
\end{enumerate}
\end{theorem}
\begin{proof}
The equivalence of 1 and 2 follows from theorems of Hsu--Wise \cite{hsu_cubulating_2015} and Haglund--Wise \cite{haglund_combination_2012}. Wise used these results in his proof of the equivalence of 1 and 3 \cite{wise_structure_2012}.
\end{proof}

Although quasiconvex subgroups of word-hyperbolic groups may not be (almost) malnormal, they enjoy a weaker algebraic property that is almost as useful, first studied by Gitik, Mitra, Rips and Sageev \cite{gitik_widths_1998}.

\begin{definition}
Let $H$ be a subgroup of a group $G$.  The \emph{width} of $H$ is the maximal cardinality of a set of distinct right-cosets
\[
\{H\gamma_i\}\subseteq H\backslash G
\]
such that $|H^{\gamma_i}\cap H^{\gamma_j}|=\infty$ for all $i,j$.
\end{definition}

In particular, note that the width of $H$ is 0 if and only if $H$ is finite, and the width of $H$ is 1 if and only if $H$ is infinite and almost malnormal.

\begin{theorem}[Gitik--Mitra--Rips--Sageev \cite{gitik_widths_1998}]\label{thm: GMRS}
If $G$ is hyperbolic and $H$ is quasiconvex in $G$ then the width of $H$ is finite.
\end{theorem}

Recall that a subgroup $H$ is \emph{separable} if it is closed in the profinite topology on $G$.  The following lemma is often useful when combined with Theorem \ref{thm: GMRS}.

\begin{lemma}\label{lem: Virtually malnormal}
If a subgroup $H$ of a torsion-free group $G$ is both separable
and has finite width, then there is a subgroup $G_0$ of finite
index in $G$ that contains $H$ and such that $H$ is malnormal in
$G_0$.
\end{lemma}

Therefore, as long as we know that $H$ is separable (as is the case for quasiconvex subgroups of virtually special groups, for instance), we may promote
finite width to malnormality, by passing to a finite-index subgroup.

The connection to 3-manifolds arises from Agol's solution to the Virtual Haken Conjecture \cite{agol_virtual_2013} which, in addition to the above theorems of Haglund, Hsu and Wise, makes essential use of the work of Kahn--Markovic \cite{kahn_immersing_2012}.

\begin{theorem}[Agol]\label{thm: VH conjecture}
If $M$ is a closed hyperbolic 3-manifold then $\pi_1M$ is hyperbolic and virtually special.
\end{theorem}

In summary, the Virtual Haken Conjecture asserts that, after passing to a finite-index subgroup, every closed hyperbolic 3-manifold admits a Haken hierarchy.  However, Agol proved the even stronger result that, after passing to a finite-index subgroup, every closed hyperbolic 3-manifold admits a quasiconvex malnormal hierarchy.

The malnormality of the hierarchy will be crucial to our argument, since we will be working in the profinite setting, where there  is no well developed notion of quasiconvexity. Malnormality, however, is a purely algebraic notion, and therefore makes perfect sense for subgroups of profinite completions.

\section{Haglund and Wise's canonical completion}\label{sec: Canonical completion}

One important element of the proof of Theorem \ref{thm: Hierarchies} is Haglund and Wise's combination theorem for virtually special cube complexes \cite{haglund_combination_2012}. Their principal tool is the \emph{canonical completion} of a quasiconvex subgroup.  In this section, we extract an algebraic property of the canonical completion, which we will then be able to apply in the profinite setting.  Our main result is as follows.

\begin{theorem}\label{thm: Canonical completion}
Let $G$ be a  hyperbolic group and the fundamental group of a
compact, virtually special cube complex, and let
$\mathcal{H}=\{H,K\}$ be a malnormal family of
quasiconvex subgroups.  For any finite-index subgroup $G_1$
of $G$ there exists a finite-index subgroup
$G_0\subseteq G_1$ and a retraction map $\rho:G_0\to
H\cap G_1$ with the following properties:
\begin{enumerate}
\item $\rho(H^\gamma\cap G_0)=1$ unless $\gamma\in
H G_1$; and \item $\rho(K^\gamma\cap G_0)=1$
for all $\gamma\in G$.
\end{enumerate}
\end{theorem}

To prove this we assemble some definitions and results of Haglund--Wise.  We refer to their paper and the references therein for definitions.

Let $X$ be a compact special cube complex, and let $A,B\subseteq X$ be subcomplexes.  Then Haglund--Wise define the \emph{wall projection} $W_X(A\to B)$, a subcomplex of $X$.  We refer to their paper for the definition, the details of which we will not need, but we note that $W_X(A\to B)$ is said to be \emph{trivial} if every loop in $W_X(A\to B)$ is null-homotopic in $X$.

\begin{theorem}\label{thm: Wall projections}
Let $A,B$ be locally convex subcomplexes of a special cube complex $X$.  There is finite-sheeted covering space $C(A,X)\to X$ to which the inclusion $A\to X$ lifts, and a retraction $r:C(A,X)\to A$.  Furthermore, if $B_0$ is the full preimage of $B$ in $C(A,X)$, then
\[
r(B_0)\subseteq W_X(B\to A)~.
\]
\end{theorem}
\begin{proof}
The existence of $C(A,X)$ and $r$ were proved in \cite{haglund_special_2008}.  The inclusion is Lemma 3.16 of \cite{haglund_combination_2012}.
\end{proof}

The covering space $C(A,X)$ is called the \emph{canonical completion} of $A$, and the map $r$ is the \emph{canonical retraction}. The above result shows that the image of the canonical retraction is controlled by the wall projection.  Ensuring that wall projections are trivial is a delicate problem, and much of the difficulty of \cite{haglund_combination_2012} lies in its resolution. In order to resolve the problem, we need to recall the definition of an \emph{elevation}.

\begin{definition}
Let $f:Y\to X$ be a continuous map of topological spaces and let $p:X'\to X$ be a covering map. Recall that the fibre product
\[
Y\times_X X'=\{(y,x')\in Y\times X'\mid f(y)=p(x')\}
\]
is naturally a covering space of $Y$ and comes equipped with a natural lift $f':Y\times_X X'\to X'$.  The restriction of $f'$ to a path component $Y'$ of $Y\times_X X'$ is called an \emph{elevation of $f$ to $X'$}.
\end{definition}

\begin{remark}
By standard covering space theory, after choosing base points there is a natural bijection between the set of double cosets
\[
f_*\pi_1Y\backslash \pi_1X/\pi_1X'
\]
and the path components of $Y\times_X X'$. If this bijection is denoted by
\[
(f_*\pi_1Y)\gamma \pi_1X'\mapsto Y_\gamma
\]
then
\[
f'_*\pi_1Y_\gamma=(f_*\pi_1Y)^\gamma\cap\pi_1X'~,
\]
where lower star means  the induced homomorphism of the
fundamental groups.

\end{remark}

We are now ready to state a theorem of Haglund and Wise, which guarantees that we can make wall projections trivial in a finite-sheeted cover.

\begin{theorem}[\cite{haglund_combination_2012}, Corollary 5.8]\label{thm: Trivial wall projections}
Let $X$ be a compact, virtually special cube complex with $\pi_1X$ word-hyperbolic.  Suppose that $A\to X$, $B\to X$ are local isometries (ie $\pi_1A,\pi_1B$ are quasiconvex subgroups of $X$) and that $\{\pi_1A,\pi_1B\}$ is a malnormal family of subgroups of $\pi_1X$.  There is a finite-sheeted covering space $A_0\to A$ such that any further finite-sheeted covering space $\overline{A}\to A_0\to A$ can be completed to a finite-sheeted covering space $\overline{X}\to X$ with the following properties:
\begin{enumerate}
\item every elevation of $A\to X$ and $B\to X$ to $\overline{X}$ is injective;
\item every elevation $A'$ of $A\to X$ to $\overline{X}$ distinct from $\overline{A}$ has trivial $W_{\overline{X}}(A'\to \overline{A})$;
\item every elevation $\overline{B}$ of $B\to X$ to $\overline{X}$ has trivial $W_{\overline{X}}(\overline{B}\to \overline{A})$.
\end{enumerate}
\end{theorem}

We will also need the simple observation that the property of trivial wall projections is preserved under passing to finite-sheeted covers.

\begin{lemma}[\cite{haglund_combination_2012}, Lemma 5.2]\label{lem: Trivial wall projections lifts}
Suppose that $X'\to X$ is a covering map of connected cube complexes, and that $A,B\subseteq X$ are connected sub-complexes. If $W_X(B\to A)$ is trivial then, for any pair of elevations $A',B'$ (of $A$ and $B$ respectively) to $X'$, $W_{X'}(B'\to A')$ is also trivial.
\end{lemma}

We can combine these to obtain the result that we need.

\begin{proof}[Proof of Theorem \ref{thm: Canonical completion}]
Let $X$ be a virtually special cube complex with
$G\cong\pi_1X$.  Since $H$ and $K$ are
quasiconvex subgroups, they are realized by local isometries of
cube complexes $A\to X$ and $B\to X$ respectively.  Let $G_1$
be any finite-index subgroup of $G$ and let $X^1$ be the
corresponding covering space of $X$.

We first prove the following: for any $\gamma\notin
H G_1$, there exists a finite-sheeted covering space
$X'\to X$ such that:
\begin{enumerate}
\item the trivial elevation $A'\to X'$ of $A\to X$ is injective;
\item the elevation $A'_\gamma\to X'$ of $A\to X$ corresponding to the double coset $H \gamma\pi_1X'$ is injective;
\item the wall projections $W_{X'}(A'_\gamma\to A')$ are trivial.
\end{enumerate}
To prove this, let $X^1\to X$ be the covering space corresponding
to $G_1$, let $A^1\to X^1$ be the elevation of $A\to X$ to
$X^1$ that corresponds to the trivial double coset, and let
$A^1_\gamma\to X^1$ be the elevation corresponding to
$H\gamma G_1$.  Since
$\{H\cap G_1,H^\gamma\cap G_1\}$ is a
malnormal pair of subgroups of $G_1$, we may apply Theorem
\ref{thm: Trivial wall projections} with $X^1$ in the role of $X$,
$A^1$ in the role of $A$ and $A^1_\gamma$ in the role of $B$ to
obtain a finite-sheeted covering space $A'\to A^1$ (such as the
one denoted by $A_0$ in the statement of the theorem) and a
corresponding `completion' to a finite covering space $X'$ of
$X^1$. Now, by item 1 of Theorem \ref{thm: Trivial wall
projections}, the maps $A'\to X'$ and $A'_\gamma\to X'$ are
injective, and by item 3, $W_{X'}(A'_\gamma\to A')$ is trivial.

Since the above covering space $X'$ depends on the element
$\gamma$, we will denote it by $X^\gamma_1$ (although note that it
actually only depends on the double coset
$H\gamma G_1$).

In exactly the same way, we may also prove the following: for any $\gamma\in G$, there exists a finite-sheeted covering space $X'\to X$ such that:
\begin{enumerate}
\item the trivial elevation $A'\to X'$ of $A\to X$ is injective;
\item the elevation $B'_\gamma\to X'$ of $B\to X$ corresponding to the double coset $K \gamma\pi_1X'$ is injective;
\item the wall projections $W_{X'}(B'_\gamma\to A')$ are trivial.
\end{enumerate}

We will denote this covering space $X'$ by $X^\gamma_2$.

Now let $\overline{X}$ be the finite-sheeted covering space of $X$ such that
\[
\pi_1\overline{X} = \bigcap_{\gamma\notin
H G_1}\pi_1X^\gamma_1 \cap
\bigcap_{\gamma\in G}\pi_1X^\gamma_1
\]
which has the property that it covers every $X^\gamma_1$ and every $X^\gamma_2$. For any $\gamma\in G$, let $\overline{A}_\gamma\to \overline{X}$ be the elevation of $A\to X$ corresponding to $H\gamma\pi_1\overline{X}$ and let $\overline{B}_\gamma\to \overline{X}$ be the elevation of $B\to X$ corresponding to $H\gamma\pi_1\overline{X}$.   Let $\overline{A}\to \overline{X}$ be the elevation of $A\to X$ corresponding to the trivial coset.  By the construction of $\overline{X}$ and by Lemma \ref{lem: Trivial wall projections lifts}, we have that $W_{\overline{X}}(\overline{A}_\gamma\to\overline{A})$ is trivial whenever $\gamma\notin H G_1$, and $W_{\overline{X}}(\overline{B}_\gamma\to\overline{A})$ for all $\gamma$.

Finally, let $G_0=\pi_1C(\overline{A},\overline{X})$ and let $\rho$ be the map induced on fundamental groups by the canonical retraction $r$.   Theorem \ref{thm: Canonical completion} now follows immediately from Theorem \ref{thm: Wall projections}.
\end{proof}

\section{Malnormality in the profinite completion}\label{sec: Profinite malnormality}

In this section we prove that, for quasiconvex subgroups of hyperbolic, virtually special groups, malnormality passes to the profinite closure.  In preparation, we need to observe that the profinite completion of a torsion-free, virtually special group is itself torsion-free.  For this we need to explain Serre's notion of a good group, which we will need in the proof.

\begin{definition} \label{good}
A group $G$ is \emph{good}, if the natural homomorphism $G\rightarrow \widehat{G}$ of the group in its profinite completion induces an isomorphism on cohomology with finite coefficients.
\end{definition}

In particular, the profinite completion of a good group of finite cohomological dimension is of finite cohomological dimension and so is torsion free.  It follows quickly from the results of \cite{grunewald_cohomological_2008} and the standard theory of virtually special groups that a torsion-free, virtually special group is good.

\begin{proposition}\label{prop: VS good}
If a group $G$ has a finite-index subgroup $G_0$ which is the fundamental group of a compact, special cube complex then $G$ is good. If $G$ is torsion-free then so is $\wh{G}$.
\end{proposition}
\begin{proof}
Let $G_0$ be a subgroup of $G$ which is the fundamental group of a compact, special cube complex.   The results of \cite{haglund_special_2008} combined with \cite[Theorem 1.4]{grunewald_cohomological_2008} imply that $G_0$ is good, and therefore $G$ is also good by \cite[Lemma 3.2]{grunewald_cohomological_2008}.  If $G$ is torsion-free then it is of finite cohomological dimension; by goodness, $\wh{G}$ is also of finite cohomological dimension and hence torsion-free.
\end{proof}

We are now ready to prove the main theorem of this section.

\begin{theorem}\label{thm: Malnormal closure}
Let $G$ be word-hyperbolic and the fundamental group of a compact, virtually special cube complex. Let $\mathcal{H}=\{H_1,\ldots,H_n\}$ be a malnormal family of quasiconvex subgroups.  Then the family $\widehat{\mathcal{H}}=\{\widehat{H}_1,\ldots,\widehat{H}_n\}$ is a malnormal family of subgroups of the profinite completion $\widehat{G}$.
\end{theorem}
\begin{proof}
It suffices to take $n=2$.  First, we prove that $\widehat{H}=\widehat{H}_1$ is a malnormal subgroup of $\widehat{G}$.

Let $\hat{\gamma}\in \widehat{G}\smallsetminus
\widehat{H}$, and suppose that
$\hat{\delta}\in\widehat{H}\cap\widehat{H}^{\hat{\gamma}}$,
so $\hat{\delta}=\hat{\epsilon}^{\hat{\gamma}}$ for
$\hat{\epsilon}\in\widehat{H}$.   Since $\widehat{H}$ is
closed, there exists a finite quotient $q:G\to Q$ whose
continuous extension $\hat{q}:\widehat{G}\to Q$ satisfies
$\hat{q}(\hat{\gamma})\notin q(H)$.  Let $G_1=\ker q$,
and let $G_0$ be the finite-index subgroup guaranteed by
Theorem \ref{thm: Canonical completion}. Let $n$ be such that
$\hat{\delta}^n\in\widehat{G}^0$.   If $\hat{\rho}$ is the
continuous extension of $\rho$ to $\widehat{G}^0$ then
\[
\hat{\delta}^n=\hat{\rho}(\hat{\delta}^n)=\hat{\rho}((\hat{\epsilon}^n)^{\hat{\gamma}}))=1
\]
where the final equality follows from item 1 of Theorem \ref{thm:
Canonical completion} by continuity using that the closure
$\overline{H G_1}$ is clopen in $\widehat{G}$ . So
$\hat{\delta}$ is torsion and therefore trivial, since
$\widehat{G}$ is torsion-free by Proposition \ref{prop: VS good}. This proves that
$\widehat{H}_1$ is malnormal.

To complete the proof that $\widehat{\mathcal{H}}$ is a malnormal
family, suppose that $\hat{\delta}=\hat{\epsilon}^{\hat{\gamma}}$,
where $\hat{\delta}\in\widehat{H}_1$ and
$\hat{\epsilon}\in\hat{H}_2$.  Let $G_1=G$, and as
before let $G_0$ be the finite-index subgroup guaranteed by
Theorem \ref{thm: Canonical completion} and let $n$ be such that
$\hat{\delta}^n\in\widehat{G}^0$.  Then, as before, we have
that
\[
\hat{\delta}^n=\hat{\rho}(\hat{\delta}^n)=\hat{\rho}((\hat{\epsilon}^n)^{\hat{\gamma}}))=1
\]
where the final inequality follows from item 2 of Theorem \ref{thm: Canonical completion} by continuity.  Again, since $\widehat{G}$ is torsion-free, we deduce that $\hat{\delta}=1$, which proves the theorem.
\end{proof}

In the statement of Theorem \ref{thm: Malnormal closure}, the group $G$ is assumed to be the fundamental group of a compact, virtually special cube complex.  In particular, $G$ is torsion-free.  However, we can weaken the hypotheses on $G$ in a small but significant way, and instead assume that $G$ is merely virtually special, meaning that it has a subgroup of finite index which is the fundamental group of a compact special cube complex.  In particular, such a $G$ may have torsion.  As mentioned above, in this context it is preferable to work with almost malnormal families of subgroups.

We can quickly deduce a similar result in this setting, which will be useful later.

\begin{corollary}\label{cor: Almost malnormal closure}
Let $G$ be word-hyperbolic and virtually special. Let $\mathcal{H}=\{H_1,\ldots,H_n\}$ be an (almost) malnormal family of quasiconvex subgroups.  Then the family $\widehat{\mathcal{H}}=\{\widehat{H}_1,\ldots,\widehat{H}_n\}$ is an (almost) malnormal family of subgroups of the profinite completion $\widehat{G}$.
\end{corollary}
\begin{proof}
Let $G_0$ be a subgroup of finite index in $G$ that is the fundamental group of a compact special cube complex.   For each $i$ let $\{\gamma_{ij}\}$ be a set of representatives for the double coset space $H_i\backslash G/G_0$. For each $i,j$, let $K_{ij}=H_i^{\gamma_{ij}}\cap G_0$.  Then $\mathcal{K}=\{K_{ij}\}$ is a malnormal family of quasiconvex subgroups $G_0$.  By Theorem \ref{thm: Malnormal closure} it follows that $\wh{\mathcal{K}}=\{\wh{K}_{ij}\}$ is a malnormal family of subgroups of the profinite completion $\wh{G}_0$.  It follows that $\wh{\mathcal{H}}$ is an (almost) malnormal family of subgroups of $\wh{G}$.
\end{proof}

Likewise, we can deduce that the notion of \emph{finite width} also passes to the profinite completion.

\begin{corollary}\label{cor: Finite width closure}
Let $G$ be word-hyperbolic and virtually special, and let $H$ be a quasiconvex subgroup of $G$.  Then $\widehat{H}$ has finite width in the profinite completion $\widehat{G}$.
\end{corollary}
\begin{proof}
Since $H$ is separable in $G$, by Lemma \ref{lem: Virtually malnormal}, we may pass to a subgroup $G_0$ of finite index in $G$ in which $H$ is malnormal. Therefore, by Corollary \ref{cor: Almost malnormal closure}, $\widehat{H}$ is malnormal in $\widehat{G}_0$.  Suppose now that $\{\widehat{H}\hat{\gamma}_i\}$ is a subset of $\widehat{H}\backslash\widehat{G}$ and $\widehat{H}^{\hat{\gamma}_i}\cap\widehat{H}^{\hat{\gamma}_j}$ is infinite for all $i$ and $j$.  Then $\widehat{H}^{\hat{\gamma}_i\hat{\gamma}_j^{-1}}\cap\widehat{H}$ is also infinite, whence $\hat{\gamma}_i\hat{\gamma}_j^{-1}\notin\widehat{G}_0$ if $i\neq j$.  Therefore, the  map
\[
\{\widehat{H}\hat{\gamma}_i\}\to\widehat{G}_0\backslash\widehat{G}\cong G_0\backslash G
\]
induced by the inclusion $\widehat{H}\to\widehat{ G}_0$ is an injection.  This completes the proof.
\end{proof}

\section{Malnormality in the relative case}

In order to deal with cusped hyperbolic manifolds, we will also need a relative versions of the results of the previous section.  A \emph{toral relatively hyperbolic} is a group that is torsion-free and hyperbolic relative to sets of finitely generated abelian subgroups.  We refer the reader to \cite{hruska_relative_2010} for a survey of the various equivalent definitions of relative hyperbolicity.

There is a notion of a \emph{relatively quasiconvex} subgroup of a relatively hyperbolic group.  Again, the reader is referred to \cite{hruska_relative_2010} for various equivalent definitions. We will also be interested in a relatively hyperbolic version of malnormality.

\begin{definition}
Suppose that a group $G$ is hyperbolic relative to a collection of \emph{parabolic} subgroups $\{P_1,\ldots,P_n\}$.  A subgroup $H$ of $G$ is called \emph{relatively malnormal} if, whenever an intersection of conjugates $H^\gamma\cap H$ is not conjugate into some $P_i$, we have $\gamma\in H$.
\end{definition}

The next theorem, which is the main result of this section, is an analogue of Theorem \ref{thm: Malnormal closure} in the toral relatively hyperbolic setting.

\begin{theorem}\label{thm: Profinite relative malnormality}
Suppose that $G$ is a virtually compact special group, which is also toral relatively hyperbolic with parabolic subgroups $P_1,\ldots,P_n$.  Let $H$ be a subgroup which is relatively malnormal and relatively quasiconvex.  Then $\wh{H}$ is also a relatively malnormal subgroup of $\wh{G}$, in the sense that $\wh{H}\cap\wh{H}^{\hat{\gamma}}$ {is conjugate into  $\wh{P}_i$} (for some $i$) whenever $\wh{\gamma}\notin\wh{H}$.
\end{theorem}

One could envisage a proof of this theorem along the lines of the techniques of Section \ref{sec: Profinite malnormality}, again using Haglund and Wise's canonical completion.  However, this would require a generalization of Theorem \ref{thm: Trivial wall projections} to the relatively hyperbolic setting, which is not currently in the literature. Instead, we will appeal to \cite[Lemma 16.13]{wise_structure_2012}, which is a relatively hyperbolic version of the Malnormal Special Quotient Theorem of Wise \cite[Theorem 12.3]{wise_structure_2012}.

\begin{theorem}\label{thm: RHMSQT}
Suppose $G$ is toral relatively word-hyperbolic and virtually compact special and $\mathcal{P}=\{P_1,\ldots,P_k\}$ is an almost malnormal family of quasiconvex subgroups of $G$. There are subgroups of finite index $K_i\subseteq P_i$ such that, for all subgroups of finite index $L_i\subseteq K_i$, the quotient
\[
G/\langle\langle L_1,\ldots,L_n\rangle\rangle
\]
is word-hyperbolic and virtually the fundamental group of a compact special cube complex.
\end{theorem}
\begin{proof}
This is a special case of \cite[Lemma 16.13]{wise_structure_2012}.
\end{proof}

We will also need to make use of a relatively hyperbolic extension of the results of Agol, Groves and Manning from \cite{agol_residual_2008}.

\begin{theorem}[Groves and Manning \cite{groves_quasiconvexity_????}]\label{thm: RHAGM}
Let $G$ be a toral relatively hyperbolic group, with
parabolic subroups $\{P_1,\ldots,P_n\}$ and let $H$ be a subgroup
which is relatively quasiconvex and relatively malnormal.   There
exist subgroups of finite index $K'_i\subseteq P_i$ (for all $i$)
such that, for all subgroups of finite index $L_i\subseteq K'_i$,
if
\[
\eta:G\to Q= G/\langle\langle L_1,\ldots,L_n\rangle\rangle
\]
is the quotient map, the quotient $Q$ is word-hyperbolic and the image $\eta(H)$ in $Q$ is quasiconvex and almost malnormal.
\end{theorem}

We will now prove Theorem \ref{thm: Profinite relative malnormality}.

\begin{proof}[Proof of Theorem \ref{thm: Profinite relative malnormality}]
Suppose that $\hat{h}\in\wh{H}$, $\hat{\gamma}\notin \wh{H}$ and $\hat{h}^{\hat{\gamma}}\in\widehat{H}$.  Since $\hat{\gamma}\notin\wh{H}$, there exists a finite quotient $q_0:G\to Q_0$ such that, when extended to the profinite completion, $\hat{q}_0(\hat{\gamma})\notin q_0(H)$ and also $\hat{q}_0(\hat{h})\neq 1$.  Let $q:G\to Q$ be any finite quotient of $G$ such that $q_0$ factors through $q$.  For each $i$, let
\[
L_i=\ker q\cap K_i\cap K'_i
\]
for $K_i$ as in Theorem \ref{thm: RHMSQT} and $K'_i$ as in Theorem \ref{thm: RHAGM}.   Now we Dehn fill to obtain
\[
\eta:G\to G/\langle\langle L_1,\ldots, L_n\rangle\rangle=:\Delta~,
\]
and let $K=\eta(H)$.  By Theorem \ref{thm: RHMSQT}, $\Delta$ is virtually special.  Note that $q_0$ factors through $\eta$, and so $\hat{\eta}(\hat{\gamma})\notin\widehat{K}$ and $\hat{\eta}(\hat{h})\neq 1$.  By Theorem \ref{thm: RHAGM}, $K$ is quasiconvex and almost malnormal in $\Delta$.

Hence, by Corollary \ref{cor: Almost malnormal closure}, $\widehat{K}$ is also almost malnormal in $\wh{\Delta}$.  But $\hat{\eta}(\hat{h})\in \widehat{K}\cap\widehat{K}^{\hat{\eta}(\hat{\gamma})^{-1}}$, and so $\hat{\eta}(\hat{h})$ has finite order, and hence is conjugate into $\hat{\eta}(\widehat{P}_i)$ for some $i$.  In particular, $\hat{q}(\hat{h})$ is conjugate into $\hat{q}(\widehat{P}_i)$ for some $i$.

Since $q$ was an arbitrarily deep finite quotient of $G$, it follows that $\hat{h}$ is conjugate into $\widehat{P}_i$ for some $i$.  This completes the proof.
\end{proof}

In the case of cusped hyperbolic manifolds, it is also an important fact that the peripheral subgroups form a malnormal family. We will need the profinite version of this fact (Lemma \ref{lem: Peripheral profinite malnormality} below).  This result was also needed in \cite{hamilton_separability_2013}, but the result stated there \cite[Lemma 4.7]{hamilton_separability_2013} is slightly weaker, as noticed independently by Gareth Wilkes and the second author.  We therefore provide a strengthened result here, with thanks to Emily Hamilton.

\begin{lemma}[Cf.\ Lemma 4.7 of \cite{hamilton_separability_2013}]\label{lem: Peripheral profinite malnormality}
Let $\Gamma$ be the fundamental group of a cusped hyperbolic manifold and let $\{P_i\}$ be conjugacy representatives of the cusp subgroups.  Then the set of their closures $\{\overline{P}_i\}$ in the profinite completion $\wh{\Gamma}$ forms a malnormal family: that is, if $\overline{P}_i\cap\overline{P}_j^{\hat{\gamma}}$ is non-trivial for some $\hat{\gamma}\in\wh{\Gamma}$ then $i=j$ and $\hat{\gamma}\in\overline{P}_i$.
\end{lemma}

Note that it is enough to consider the case of two cusps, so we adopt the notation $P=P_1$ and $Q=P_2$.  We first consider the intersection $\overline{P}\cap\overline{Q}^{\hat{\gamma}}$. The next lemma is  a slight modification of \cite[Lemma 4.6]{hamilton_separability_2013}.

\begin{lemma}\label{lem: Congruence quotients}
Let $M = {\Bbb H}^3 / \Gamma$ be a cusped hyperbolic $3$-manifold of finite volume.  
Let $P$ and $Q$ be non-conjugate cusp subgroups of $\Gamma$.  
Then there exists a positive integer $n$ with the following property.  For each integer $m \geq n$, there exist finite fields $F_1$ and $F_2$
and group homomorphisms $f_1: \Gamma \rightarrow \mathrm{PSL}(2, F_1)$ and $f_2: \Gamma \rightarrow \mathrm{PSL}(2, F_2)$ such that:

\begin{enumerate}
\item the image of $P$ under $f_1 \times f_2$ is isomorphic to $\mathbb{Z} / m \mathbb{Z} \times \mathbb{Z} / m \mathbb{Z};$
\item for every element $p\in P$ and every $i$, if $f_i(p)$ is non-trivial, then $\tr f_i(p)\neq \pm2$;
\item for every element $q\in Q$ and every $i$, $\tr f_i(q)=\pm2$.
\end{enumerate}
In particular, if every $f_i(p)$ is conjugate into $f_i(Q)$, 
then $(f_1 \times f_2)(p)$ is trivial.
\end{lemma}

\begin{proof}Since $P$ and $Q$ are non-conjugate, they correspond to distinct cusps of $M$.  Let $T$ be the cusp of $M$ corresponding to $P$.  
The group $P$ is free abelian of rank $2$.
By Thurston's hyperbolic Dehn surgery theorem, there exist a basis $\{ p_1, p_2 \}$ of $P$ and complete hyperbolic $3$-manifolds $M_1$ and $M_2$ of finite volume, 
obtained by Dehn surgery on $M$ along $T$, such that 
if $$\phi_1: \Gamma
\rightarrow \pi_1(M_1) \ \mathrm{and} \ \phi_2: \Gamma \rightarrow \pi_1(M_2)$$ are the homomorphisms induced by inclusion, then $\phi_1(p_1)$ is a
loxodromic isometry of $\pi_1(M_1)$, $\phi_1(p_2)$ is trivial, 
$\phi_2(p_1)$ is trivial and $\phi_2(p_2)$ is a loxodromic isometry of
$\pi_1(M_2)$. 

Since $M_1$ has finite volume and $\phi_1(p_1)$ is loxodromic, there exists
a discrete, faithful representation $$\rho_1: \pi_1(M_1) \rightarrow 
\mathrm{PSL}(2,\mathbb{C})$$ such that $\rho_1(\pi_1(M_1)) \subset \mathrm{PSL}(2,L_1)$, for some number field $L_1$,
and $$\rho_1(\phi_1(p_1)) =  
\pm \begin{pmatrix}
\omega & 0 \\
0 & \omega^{-1} \\ \end{pmatrix}, \  \vert \omega \vert \neq 1. $$   
Let $R_1$ be the ring in $L_1$ generated by
the coefficients
of the generators of $\rho_1(\pi_1(M_1))$ over $\Bbb Z$.
Then $\rho_1(\pi_1(M_1)) \subset \mathrm{PSL}(2, R_1)
\subset \mathrm{PSL}(2, L_1)$.  By Corollary 2.5 of \cite{hamilton_finite_2005}, there exists a positive integer $n$ with the following property.  For each integer $m \geq n$, there exist a finite field $F_1$ and a ring
homomorphism $\eta_1: R_1 \rightarrow F_1$ such that the multiplicative order of $\eta_1(\omega)$ is equal to $2m$.  This ring homomorphism
induces a group homomorphism 
$$\psi_1:  \rho_1(\pi_1(M_1)) \hookrightarrow \mathrm{PSL}(2, R_1) \rightarrow
\mathrm{PSL}(2, F_1).$$  
Let $$f_1: \Gamma \rightarrow \mathrm{PSL}(2,F_1)$$ denote the composition
$\psi_1 \circ  \rho_1 \circ \phi_1$.  Then $f_1(p_1)$ has order $m$ and 
$f_1(p_2)$ is trivial.  For every element $q \in Q$, $\rho_1\circ\phi_1(q)$ is 
parabolic.  Therefore, $\tr f_1(q)=\pm2$.  For every element $p \in P$,
$$\rho_1(\phi_1(p)) =  
\pm \begin{pmatrix}
\omega^k & 0 \\
0 & \omega^{-k} \\ \end{pmatrix}, \ \mathrm{for \ some} \ k \in \mathbb{Z}. $$ If $\eta_1(\omega^k + \omega^{-k}) = 2$, then 
$\eta_1(\omega^k) = \eta_1(\omega^{-k}) = 1$.  If 
$\eta_1(\omega^k + \omega^{-k}) =  -2$, then 
$\eta_1(\omega^k) = \eta_1(\omega^{-k}) =  -1$.  Therefore, if
$f_1(p)$ is non-trivial, then $\tr f_1(p)\neq \pm2$. 

In a similar way, we can choose $n$ such that for each integer $m \geq n$ 
there exist a finite field $F_2$
and a group homomorphism $$f_2: \Gamma \rightarrow \mathrm{PSL}(2, F_2)$$ such that
$f_2(p_1)$ is trivial, $f_2(p_2)$ has order $m$, for every $q \in Q$,
$\tr f_2(q) = \pm2$, and, for every  $p \in P$, if $f_2(p)$ is non-trivial, then $\tr f_2(p) \neq \pm2$. The homomorphisms $f_1$ and $f_2$
then satisfy the three conditions above.
\end{proof}

We can now easily prove that the peripheral subgroups form a profinite malnormal family.

\begin{proof}[Proof of Lemma \ref{lem: Peripheral profinite malnormality}]
Consider first the intersection $\overline{P}\cap\overline{Q}^{\hat{\gamma}}$.  Let $\hat{p}$ be an element of the intersection $\overline{P}\cap\overline{Q}^{\hat{\gamma}}$.  Let $f$ be an arbitrary homomorphism from $\Gamma$ to a finite group, and let $f_0$ be the restriction of $f$ to $P$.  Choose $f_1,f_2$ as in Lemma \ref{lem: Congruence quotients} so that $f_0$ factors through $(f_1\times f_2)|_P$, and extend them by continuity to homomorphisms $\hat{f}_i$ from the profinite completion.  Choose $p\in P$ so that $f_i(p)=\hat{f}_i(\hat{p})$ for all $i$.  Since $\hat{p}$ is conjugate into $\overline{Q}$, item 3 implies that $\tr f_i(p)=\tr \hat{f}_i(\hat{p})=2$ for every $i$. Item 2 then implies that $f_i(p)=1$ for every $i$, and so by item 1, $\hat{f}_0(\hat{p})=f_0(p)=1$.  Thus, every finite quotient of $\Gamma$ kills $\hat{p}$, and so $\hat{p}=1$ by the definition of the profinite completion.  We have shown that the intersection $\overline{P}\cap\overline{Q}^{\hat{\gamma}}$ is trivial.

We now consider the intersection $\overline{P}\cap\overline{P}^{\hat{\gamma}}$, and suppose that $\hat{\gamma}\notin\overline{P}$. Since $\overline{P}$ is closed, there is a subgroup $\Gamma_0$ of finite index in $\Gamma$ that contains $P$ but such that $\wh{\Gamma}_0\subseteq\wh{\Gamma}$ does not contain $\hat{\gamma}$.  Let $\hat{\gamma}=\gamma\hat{\gamma}_0$, where $\gamma\in\Gamma\smallsetminus\Gamma_0$ and $\hat{\gamma}_0\in\wh{\Gamma}_0$.  Then $\Gamma_0$ is the fundamental group of a closed hyperbolic manifold with non-conjugate cusp subgroups $P$ and $Q=P^\gamma\cap\Gamma_0$ (since $P$ is malnormal in $\Gamma$).  By the argument of the previous paragraph applied to $\Gamma_0$,  the intersection $\overline{P}\cap\overline{Q}^{\hat{\gamma}_0}$ is trivial.  Since $Q$ is of finite index in $P^\gamma$, it follows that the intersection $\overline{P}\cap\overline{P}^{\hat{\gamma}}$ is finite.  But 3-manifold groups are torsion-free and good, whence their profinite completions are also torsion-free.  In particular, $\overline{P}\cap\overline{P}^{\hat{\gamma}}$ is trivial.
\end{proof}

\section{Profinite trees}\label{sec: Profinite trees}

In order to prove our main theorems, our strategy is to promote the action of a group on a tree to an action of its profinite completion on a \emph{profinite tree}. In this section, we recall the necessary elements of the theory of profinite trees.

A graph $\Gamma$ is a disjoint  union $E(\Gamma) \cup V(\Gamma)$
with two maps $d_0, d_1 : \Gamma \to V(\Gamma)$ that are the
identity on the set of vertices $V(\Gamma)$.  For an element $e$ of
the set of edges  $E(\Gamma)$, $d_0(e) $ is called the initial and
$d_1(e) $ the terminal vertex of $e$.

\begin{definition}
A \emph{profinite graph} $\Gamma$ is a graph such that:
\begin{enumerate}
\item $\Gamma$ is a profinite space (i.e.\ an inverse limit of finite
discrete spaces);
\item $V(\Gamma)$ is closed; and
\item the maps $d_0$ and $d_1$
are continuous.
\end{enumerate}
Note that $E(\Gamma)$ is not necessary closed.
\end{definition}

By \cite[Prop.~1.7]{zalesskii_subgroups_1988} every profinite
graph $\Gamma$ is an inverse limit of finite quotient graphs of
$\Gamma$.

For a profinite space $X$  that is the inverse limit of finite
discrete spaces $X_j$, $[[\widehat{\mathbb{Z}} X]]$ is the inverse
limit  of $ [\widehat{\mathbb{Z}} X_j]$, where
$[\widehat{\mathbb{Z}} X_j]$ is the free
$\widehat{\mathbb{Z}}$-module with basis $X_j$. For a pointed
profinite space $(X, *)$ that is the inverse limit of pointed
finite discrete spaces $(X_j, *)$, $[[\widehat{\mathbb{Z}} (X,
*)]]$ is the inverse limit  of $ [\widehat{\mathbb{Z}} (X_j, *)]$,
where $[\widehat{\mathbb{Z}} (X_j, *)]$ is the
$\widehat{\mathbb{Z}}$-vector space with basis $X_j \setminus \{ *
\}$ \cite[Chapter~5.2]{ribes_profinite_2010}.

For a profinite graph $\Gamma$ define the pointed space
$(E^*(\Gamma), *)$ as  $\Gamma / V(\Gamma)$ with the image of
$V(\Gamma)$ as a distinguished point $*$, and denote the image of $e\in E(\Gamma)$ by $\bar{e}$.  By definition  a profinite
tree  $\Gamma$ is a profinite graph with a short exact sequence
$$
0 \to [[\widehat{\mathbb{Z}}((E^*(\Gamma), *)]]
\stackrel{\delta}{\rightarrow} [[\widehat{\mathbb{Z}} V(\Gamma)]]
\stackrel{\epsilon}{\rightarrow} \widehat{\mathbb{Z}} \to 0
$$
where $\delta(\bar{e}) = d_1(e) - d_0(e)$ for every $e \in E(\Gamma)$ and $\epsilon(v) = 1$ for every $v \in V(\Gamma)$.  If $v$  and $w$ are elements  of a profinite tree   $T$, we denote by $[v,w]$ the smallest profinite subtree of $T$ containing $v$ and $w$ and call it geodesic.

By definition a profinite group $G$ acts on a profinite graph
$\Gamma$ if  we have a continuous action of $G$ on the profinite
space $\Gamma$ that commutes with the maps $d_0$ and $d_1$.

If a profinite group $G$ acts on a profinite tree $T$ then by
\cite[Lemma 1.5]{zalesskii_profinite_1990} there exists a minimal
$G$-invariant subtree $D$ of $T$ (note that in the classical
Bass--Serre theory there is another possibility, namely the group
can have an invariant end; this possibility is missing in the
profinite case because of the compactness of $T$); moreover $D$ is
unique if it is not a vertex. If $D$ is finite then it is a vertex
and so $G$ stabilizes a vertex. Therefore, if $G$ does not
stabilize a vertex, $D$ is infinite.

We state now the general Tits-alternative-type result from
\cite{zalesskii_profinite_1990} that is the key result for our use.

\begin{theorem}\label{profinite} (\cite{zalesskii_profinite_1990}) Let $H$ be a profinite group acting  on a profinite tree $T$.
Suppose $H$ does not possess a non-abelian free pro-$p$ subgroup
for every prime $p$. Then either $H$ stabilizes a vertex  or there
exists a unique infinite minimal $H$-invariant subtree $D$ of $T$
such that the quotient group $L=H/K$ modulo the kernel of the
action on $D$ is soluble and isomorphic to one of the following
groups:
\begin{enumerate}
\item $L\cong\Z_\pi\rtimes\Z_\rho$ where $\pi$ and $\rho$ are
disjoint sets of primes;
\item $L$ is a profinite dihedral
group $\Z_\pi\rtimes C_2$;
\item $L$ is a profinite Frobenius
group $\Z_\pi\rtimes C_n$, i.e.\ the order of every prime divisor
$p$ of $n$ divides $q-1$ for some $q\in\pi$ and the centralizers
of nonidentity elements of $C_n$ coincide with $C_n$.
\end{enumerate}
\end{theorem}

Note that the group of automorphisms $Aut(\Z_\pi)$ coincides with
the group of units of the ring $\Z_\pi=\prod_{p\in \pi} \Z_p$ and
so is the direct product $\prod_{p\in \pi} \Z_p^*$ of groups of
units of $\Z_p$. Remark also that $\Z_p^*\cong \Z_p\times C_{p-1}$
for $p\neq 2$ and $\Z_2^*\cong \Z_2\times C_2$.

\medskip

\begin{remark}\label{cases} If in the notation of Theorem \ref{profinite}
$H$ acts $k$-acylindrically for some natural $k$ (see Definition \ref{def: Profinite acylindrical} below) then $K=1$. If in
addition $H$ is torsion free we have just the first case of a
projective soluble group $H\cong\Z_\pi\rtimes\Z_\rho$, where $\pi$
and $\rho$ are disjoint sets of primes. In this case any torsion-free profinite group containing $H$ as an open subgroup has a similar structure.\end{remark}

\section{Graphs of groups and their profinite analogues}\label{sec: Profinite gogs}

In this section, we describe the basic theory of profinite graphs
of groups in the case when underlying graph is finite (see Section
3  \cite{zalesskii_subgroups_1988}). We shall need only this case
here, in the general case the definitions and terminology are much
more involved (see \cite{zalesskii_fundamental_1989}).

Let $\Delta$ be a connected finite graph.  A \emph{graph of groups}
$(\G,\Delta)$ over $\Delta$ assigns a group $\G(m)$
to each $m\in \Delta$, and monomorphisms $\partial_i:
\G(e)\longrightarrow \G(d_i(e))$ for each edge $e\in E(\Delta)$.
If each  $\G(m)$ is a profinite group and the monomorphisms
$\partial_i$ are continuous, we say that $(\G,\Delta)$ is a \emph{graph
of profinite groups}.

The \emph{abstract fundamental group}
\[
\Pi^{abs}= \Pi_1^{abs}(\G,\Delta)
\]
of the graph of groups $(\G,\Delta)$ can be defined by means of a universal property. Fix $T$ a maximal subtree of $\Delta$.  Then $\Pi^{abs}$ is an abstract group equipped with a collection of homomorphisms
$$
\nu_m: \G(m)\longrightarrow \Pi^{abs}\quad (m\in \Delta),
$$
and a  map $E(\Delta) \longrightarrow  \Pi^{abs}$, denoted $e\mapsto t_e$  ($e\in E(\Delta)$), such that
$t_e=1$ if $e\in E(T)$, and
$$(\nu_{d_0 (e)}\partial_0)(x)= t_e(\nu_{d_1 (e)}\partial_1)(x)t_e^{-1},\quad  \forall x\in \G(e), \ e\in E(\Delta)~.
$$

The defining universal property of $\Pi^{abs}$ is then as follows.

\medskip
\noindent {\it Whenever} one has the following data:
\begin{itemize}
\item $H$ an abstract group,
\item $\beta_m: \G(m)\longrightarrow H\quad (m\in \Delta)$ a
collection of homomorphisms, and
\item a map $e\mapsto s_e\in H$ ($e\in E(\Delta)$)  with $s_e=1$, if $e\in E(T)$,
satisfying
\[
(\beta_{d_0 (e)}\partial_0)(x)= s_e(\beta_{d_1
(e)}\partial_1)(x)s_e^{-1}, \forall x\in \G(e), \ e\in E(\Delta),
\]
\end{itemize}

\noindent {\it then}  there exists a unique homomorphism $\delta :
\Pi^{abs}\longrightarrow  H$ such that $\delta(t_e)= s_e$  $(e\in
E(\Delta))$, and for each $m\in \Delta$ the diagram

\medskip
$$\xymatrix{&
\Pi ^{abs}  \ar[dd]^\delta   \\  \G(m)  \ar[ru]^{\nu_m}
\ar[rd]_{\beta_m }\\ &H }$$

\medskip
\noindent commutes.

\medskip
In \cite[Chapter I, Definition  7.3  and Corollary
7.5]{dicks_groups_1980}, and in \cite[Part I,  Sections 5.1 and
5.2]{serre_arbres_1977} the fundamental group $\Pi^{abs}$ is
defined explicitly in terms of generators and relations; there it
is also proved that the definition given above is independent of
the choice of the maximal subtree $T$, and furthermore it is
proved that the homomorphisms $\nu_m: \G(m)\longrightarrow
\Pi^{abs}$ are injective for every $m\in \Delta$. We use the
notation $\Pi^{abs}(m) = {\rm Im}(\nu_m)$; so $\Pi^{abs}(m)\cong
\G(m)$, for $m\in \Delta$.

The definition of the  \emph{profinite fundamental group}
$$\Pi = \Pi_1 (\G,\Delta)$$
of a graph $(\G,\Delta)$ of profinite groups over a finite graph
$\Delta$ is formally as above: one simply  assumes that all the
conditions take place in the category of profinite groups, i.e.\ all groups involved are profinite  and all homomorphisms are assumed to be continuous.
 The explicit  construction of $\Pi $ is as follows (see (3.3) in
 \cite{zalesskii_subgroups_1988}). Consider the family
 \[
 {\cal N}=\{N\triangleleft_f \Pi_1^{abs}(\G,\Delta)\mid N\cap
 G(v)\triangleleft_o \G(v)\}
 \]
 of all normal subgroups of finite index whose intersection with vertex groups
 are open. The profinite fundamental group $\Pi_1 (\G,\Delta)$ is
 just
 $$
 \Pi_1(\G,\Delta)=\lim\limits_{\displaystyle\longleftarrow\atop N\in
 {\cal N}} \Pi_1^{abs} (\G,\Delta)/N
 $$
which is to say the completion of $\Pi_1^{abs}(\G,\Delta)$ with
respect to the topology defined by ${\cal N}$.

There is one important difference with the abstract case: in the profinite setting, the canonical homomorphisms  $\nu_m: \G(m)\longrightarrow \Pi$ $(m\in \Delta)$ are not embeddings in general  (cf.\ Examples 9.2.9 and 9.2.10 in \cite{ribes_profinite_2010}). We use the
 notation $\Pi(m) = {\rm Im}(\nu_m)$,  for $m\in \Delta$.

Associated with the graph of groups $(\G, \Delta)$ there is a corresponding  {\it standard graph} (or universal covering graph)
\[
S^{abs}=\bigcup_{m\in\Delta}
\Pi^{abs}/\Pi^{abs}(m)~.
\]
The set of vertices of $S^{abs}$ is given by
\[
V(S^{abs})=\bigcup_{v\in V(\Delta)}
\Pi^{abs}/\Pi^{abs}(v)~,
\]
and the incidence maps of $S^{abs}$ are given by the following formulae:
\begin{eqnarray*}
d_0 (g\Pi^{abs}(e)) & =  & g\Pi^{abs}(d_0(e)) ~;\\
 d_1(g\Pi^{abs}(e)) & = & gt_e\Pi^{abs}(d_1(e))~ (e\in E(\Delta))~.
\end{eqnarray*}
In fact $S^{abs}$  is a tree, usually called the \emph{Bass--Serre
tree} (cf.\ [\cite{dicks_groups_1980}, Chapter I, Theorem 7.6] or
[\cite{serre_arbres_1977}, part I, Section 5.3]). There is a
natural left action of $\Pi^{abs}$ on $S^{abs}$, and clearly
$\Pi^{abs}\backslash S^{abs}= \Delta$.

Analogously, there is a profinite standard graph   $S =\bigcup \Pi
/\Pi (m)$ associated with  a graph of profinite groups $(\G,
\Delta)$,  with the space of vertices and edges and with incidence
maps defined as above. In fact, $S$ is a profinite tree (cf.
Proposition 3.8 \cite{zalesskii_subgroups_1988}), and $\Pi$ acts
continuously on $S$ with $\Pi\backslash S= \Delta$.

Given an abstract graph of groups $(\G,\Delta)$, there is a naturally associated graph of profinite groups $(\bar{\G},\Delta)$.  Under certain natural hypotheses, we shall further see that $(\bar{\G},\Delta)$ is closely related to $(\G,\Delta)$.

First, suppose that $\Pi^{abs}=\Pi^{abs}(\G,\Delta)$ is residually finite, and let $\Pi$ be the profinite completion of $\Pi^{abs}$.  For each $m\in \Delta$, the profinite topology of $\Pi^{abs}$ induces on $\Pi^{abs}(m)$ a certain profinite topology (which is not necessarily its full profinite topology) and so on $\G(m)$. Define $\bar \G(m)$ to be the completion of $\G(m)$ with respect to this topology. Then the monomorphisms $\partial_i:\G(e)\longrightarrow \G(d_i(e))$ induce continuous monomorphisms which are again denoted by $\partial_i: \bar \G(e) \longrightarrow \bar \G(d_i(e)) $ $(i=0,1)$. We have then a graph $(\bar\G, \Delta)$ of profinite groups over $\Delta$. The canonical injection  $\G(m)\longrightarrow \Pi^{abs}$ induces an injection  $\bar \G(m)\longrightarrow \Pi$ \  $(m\in \Delta)$; furthermore, if we denote by $\Pi(m)$ the image of $\bar \G(m)$ on $\Pi$  under this injection, then $\Pi(m)= \overline{\Pi^{abs}(m)}$, the closure of $\Pi^{abs}(m)$ in $\Pi$.

   Clearly
$$\partial_0(g)= t_e\partial_1(g)t_e^{-1}\  (g\in \bar \G  (e) , e\in E(\Delta))$$  in $\Pi$  (there a certain abuse
of notation here, as we are identifying   $\bar \G  (v)$ with its
image  in $\Pi$, and similarly we are denoting both the original
elements $t_e$   ($e\in  E(\Delta)$)   and their images in $\Pi$, which is justified since with our assumptions $\Pi^{abs}$ injects into $\Pi$).  Furthermore, one checks immediately the following result.

\begin{proposition}
 The profinite completion
$\Pi$  of  $\Pi ^{abs}$ is  the fundamental profinite group $\Pi_1
(\bar \G, \Delta)$ of the graph of profinite
 groups $(\bar \G, \Delta)$. The canonical homomorphisms $\bar \G(m) \longrightarrow \Pi=\Pi_1 (\bar \G,
\Delta)$ are injective $(m\in   \Delta)$.
\end{proposition}

We make a further assumption, namely that $(\G,\Delta)$ is \emph{sufficient}.

\begin{definition}
A graph of groups $(\G,\Delta)$ is \emph{sufficient} if:
\begin{enumerate}
\item the fundamental group $\Pi^{abs}=\Pi^{abs}_1(\G,\Delta)$ is
residually finite; and \item for each $m\in \Delta$, $\Pi^{abs}
(m)$ is closed in the profinite topology of  $\Pi^{abs}$ (or,
equivalently, $\Pi(m)\cap \Pi^{abs}= \Pi^{abs}(m)$).
\end{enumerate}
\end{definition}

Consider the natural morphism of graphs
\[
\varphi:S ^{abs}\longrightarrow S
\]
which on vertices and edges is
\[
g\Pi^{abs}(v)\mapsto g\Pi (v), \quad g\Pi^{abs}(e)\mapsto g\Pi (e)\quad (g\in \Pi^{abs}, v\in V(\Delta), e\in E(\Delta))~.
\]
If $(\G,\Delta)$ is sufficient then $\varphi$ is an injection of graphs;  we think of $S ^{abs}$ as a subgraph of $S$. Moreover it is clear that $S ^{abs}$ is dense in $S$.   We collect all of this in the following proposition.

\begin{proposition}
 Let $(\G, \Delta)$  be a sufficient graph of abstract groups over a finite connected graph $\Delta$, and consider the graph
$(\bar \G, \Delta)$ of profinite groups over $\Delta$ such that each $\bar \G(m)$ is the completion of $\G(m)$ with respect to the
topology induced by the profinite topology of $\Pi^{abs}$. Then the standard tree
 $S^{abs}=S^{abs}(\G, \Delta)$ of the graph of groups $(\G, \Delta)$ is canonically embedded in the standard profinite tree $S=S(\bar \G, \Delta)$ of the graph of profinite groups $(\bar\G, \Delta)$,  and $S^{abs}$  is dense in $S$.
\end{proposition}

The situation is particularly attractive if the profinite topology induced on the vertex and edge groups $\G(m)$ coincides with their intrinsic profinite topology.

\begin{definition}
Suppose that $(\G,\Delta)$ is sufficient and that $\Pi^{abs}$ is the fundamental group.  If, furthermore, for each $m\in\Delta$, each finite-index subgroup $H$ of $\G(m)$ is closed in the profinite topology on $\Pi^{abs}$, then $(\G,\Delta)$ is called \emph{efficient}.
\end{definition}

In this case, the vertex and edge stabilizers in the profinite tree are conjugate to the profinite completions of the vertex and edge stabilizers in the standard tree. Results of Haglund and Wise imply that the graphs of groups arising in quasiconvex hierarchies of virtually special groups are efficient \cite{haglund_special_2008}.

\begin{definition}
We say that $(\G,\Delta)$ is \emph{of dihedral type} if $\Pi^{abs}$
has a normal subgroup $V$ contained in a vertex group
 such that $\Pi^{abs}/N$ is either infinite dihedral or cyclic.
\end{definition}

We shall need the following proposition for the proof of Theorem \ref{thm: Profinite Seifert conjecture}.

 \begin{proposition}\label{prop:normal}
 Let $\Pi^{abs} = \Pi_1^{abs} (\G,\Delta)$ be the
 fundamental group of an efficient finite graph of  groups
 $(\G,\Delta)$ not of dihedral type and $K$ a closed normal subgroup of $\widehat\Pi=\Pi_1 (\bar\G,\Delta)$ having no non-abelian free profinite subgroup. Then $K$ is contained in the completion of a vertex group of $\wh\Pi$.
 \end{proposition}

 \begin{proof} Use a decomposition $$\widehat\Pi=\lim\limits_{\displaystyle\longleftarrow\atop N
 } \Pi_1^{abs} (\G,\Delta)/N,$$ where $N$ ranges over normal subgroups of finite index in $\Pi^{abs}$. For any such $N$ define a
 graph of group $(\G_N,\Delta)$ replacing $\G(m)$ by $\G(m)N/N$
 for all $m\in \Delta$. Note that
 $\Pi_1(\G_N,\Delta)=\widehat{\Pi_1^{abs}(\G_N,\Delta)}$ and so
 contains a free profinite group as an open subgroup. It is clear also that it is a quotient group of $\widehat\Pi$ and $$\widehat\Pi=\lim\limits_{\displaystyle\longleftarrow\atop N
 } \Pi_1 (\G_N,\Delta).$$
 We denote by $K_N$ the image of $K$ in $\Pi_1 (\G_N,\Delta)$.
  Since $(\G,\Delta)$ is not
 of dihedral type, $(\G_N,\Delta)$ is not of dihedral type for some $N$ as
 well. Hence $\widehat{\Pi_1^{abs}(\G_N,\Delta)}$ contains an open
 non-abelian
 free profinite subgroup $F$. Then by \cite[Theorem 8.6.6]{ribes_profinite_2010} $K_N\cap F=1$ and so $K_N$ is finite. Hence by \cite[Theorem 3.10]{zalesskii_subgroups_1988} $K_N$ is contained in
 a vertex group of $\Pi_1(\G_N,\Delta)$. It follows then from the
 decomposition $$\widehat\Pi=\lim\limits_{\displaystyle\longleftarrow\atop N
 } \Pi_1 (\G_N,\Delta)$$ that $K$ is contained in  a vertex group of $\Pi_1
 (\G,\Delta)$ as needed. \end{proof}

\section{Virtually special hyperbolic groups}

In this section we prove Theorem \ref{thm: Hyp special}.   By
Theorem \ref{thm: Hierarchies}, a hyperbolic, virtually special
group $\Gamma$ has a subgroup $\Gamma_0$ of finite index that
admits a malnormal quasiconvex hierarchy.  (Note that this
statement does not need the most difficult parts of Theorem
\ref{thm: Hierarchies}; it is a consequence of the fact that
quasiconvex subgroups of hyperbolic groups have finite width,
together with the fact that hyperplane stabilizers in virtually
special groups are separable.)  Our proof of Theorem \ref{thm: Hyp
special} will be by induction on the length of the hierarchy for
$\Gamma_0$.

Recall the following standard definition.

\begin{definition}\label{def: Acylindrical}
The action of a group $\Gamma$ on a tree $T$ is said to be \emph{$k$-acylindrical}, for $k$ a constant, if the set of fixed points of $\gamma$ has diameter at most $k$ whenever $\gamma\neq 1$.
\end{definition}

We may make the analogous definition in the profinite setting.

\begin{definition}\label{def: Profinite acylindrical}
The action of a profinite group $\widehat{\Gamma}$ on a profinite tree $T$ is said to be \emph{$k$-acylindrical}, for $k$ a constant, if the set of fixed points of $\gamma$ has diameter at most $k$ whenever $\gamma\neq 1$.
\end{definition}

A malnormal hierarchy for $\Gamma_0$ implies that $\Gamma_0$ has a 1-acylindrical action on a tree.  This result carries over to the profinite setting.

\begin{lemma}\label{lem: Acylindrical action}
Suppose that $\Pi$ is the (profinite) fundamental group of a graph of profinite groups $(\bar{\G},\Delta)$ with one edge $e$, and suppose that the edge group $\Pi(e)$ is malnormal in $\Pi$.  Then the action of $\Pi$ on the standard tree $S$ is 1-acylindrical.
\end{lemma}
\begin{proof}
For an action to be 1-acylindrical means that the intersection of the stabilizers of any two distinct edges is trivial.  The hypothesis that there is only one edge $e$ means that any distinct pair of edges of $S$ correspond to distinct left-cosets $g\Pi(e)$ and $h\Pi(e)$, and their stabilizers are the conjugates $g\Pi(e)g^{-1}$ and $h\Pi(e)h^{-1}$.  Malnormality implies that the intersection of these is trivial.
\end{proof}

We are now ready to prove that closed subgroups of the profinite completions of virtually special hyperbolic groups satisfy a Tits alternative.

\begin{proof}[Proof of Theorem \ref{thm: Hyp special}]
Put $H_0=\widehat\Gamma_0\cap H$, where $\Gamma_0$ is the subgroup from Theorem \ref{thm: Hierarchies}. By Proposition \ref{prop: VS good}, $\widehat \Gamma_0$ is torsion free.  We prove that $H_0\cong \Z_\pi\rtimes \Z_\rho$, with $\pi\cap \rho=\emptyset$, by induction on the length of the malnormal hierarchy. By Theorem \ref{thm: Malnormal closure} and Lemma \ref{lem: Acylindrical action}, $\widehat{\Gamma}_0$ acts 1-acylindrically on a profinite tree and therefore so does $H_0$. By the inductive hypothesis, the vertex stabilzers in $H_0$ have the claimed structure.  If $H_0$ stabilizes a vertex by the induction hypothesis we are done. Otherwise, by Remark \ref{cases}, $H_0$ has the claimed structure.
\end{proof}

Theorem \ref{thm: Profinite hyperbolization} follows as a quick consequence.

\begin{proof}[Proof of Theorem \ref{thm: Profinite hyperbolization}]
By Theorem \ref{thm: VH conjecture}, if $M$ is a closed hyperbolic 3-manifold then $\pi_1M$ is virtually special, and hence Theorem \ref{thm: Hyp special} applies to show that there are no $\widehat{\Z}^2$ subgroups.  Conversely, if $\wh{\pi_1M}$ contains no subgroups isomorphic to $\wh{\Z}^2$ then, by Hamilton's result that abelian subgroups of 3-manifold groups are separable \cite{hamilton_abelian_2001}, $\pi_1M$ contains no subgroups isomorphic to $\Z^2$ and so, by the Hyperbolization theorem, $M$ is hyperbolic.
\end{proof}

\section{Seifert fibred 3-manifolds}\label{sec: Seifert fibred}

This section is devoted to characterizing certain features of Seifert fibred 3-manifolds using the profinite completions of their fundamental groups.  We start by proving Theorem \ref{thm: Profinite Seifert conjecture}, which  shows that the profinite completion of the fundamental group distinguishes Seifert fibred 3-manifolds among all 3-manifolds.

\begin{proof}[Proof of Theorem \ref{thm: Profinite Seifert conjecture}]
Suppose that $M$ is closed, orientable and irreducible.  If $M$ is Seifert fibred then $\pi_1M$ has a cyclic normal subgroup $Z$.   Since every finitely generated subgroup of $\pi_1M$ is separable \cite{scott_subgroups_1978,scott_correction_1985} $\wh{\pi_1M}$ has a procyclic normal subgroup, as claimed.

Suppose therefore that $M$ is not Seifert fibred.  If $M$ is a
Sol-manifold, then $\pi_1M\cong\Z^2\rtimes\Z$, where the action of
$\Z=\langle z\rangle$ on $\Z^2$ is Anosov, in the sense that no non-trivial power of $z$ fixes a non-trivial element of $\Z^2$. This induces
$\wh{\pi_1M}\cong\wh{\Z}^2\rtimes\wh{\Z}$.

We next show that $z$ does not fix a non-trivial element of
$\wh{\Z}^2$. Since $\wh{\Z}^2=\prod_p \Z_p^2$ it suffices to show
that $z$ does not fix a nontrivial vector in $\Z_p^2$. We show in
fact that $z$ does not fix a vector in $\Q_p^2$. Indeed, if it
does then $det(z-I)=0$ and so $z$ has a non-zero fixed point in
$\Q^2$ and hence in $\Z^2$, a contradiction.

Thus any procyclic normal subgroup of
$\wh{\pi_1M}\cong\wh{\Z}^2\rtimes\wh{\Z}$ intersects $\wh{\Z}^2$ trivially
 and so is in the kernel of the natural homomorphism
$\widehat{\langle z\rangle} \to Aut(\wh{\Z}^2)\cong
GL_2(\widehat\Z)$. Therefore to conclude that $\wh{\pi_1M}$ does
not have procyclic normal subgroups in this case, it suffices to
observe now that the closure of $\langle z\rangle$ in
$GL_2(\widehat\Z)$ is isomorphic to $\widehat \Z$. To see this it
is enough  to show that all Sylow subgroups of $\widehat{\langle
z\rangle}$ are infinite. As $GL_2(\widehat\Z)= \prod_p GL_2(\Z_p)$
and $GL_2(\Z)$ embeds in the virtually pro-$p$ group $GL_2(\Z_p)$,
the closure of $\langle z\rangle$ in $GL_2(\Z_p)$ has infinite
Sylow $p$-subgroup for each $p$ and therefore so does
$\widehat{\langle z\rangle}$.

If $M$ is non-geometric or hyperbolic, then $\wh{\pi_1M}$ acts acylindrically on a profinite tree.   (In the hyperbolic case, this follows from the results of the previous section. In the non-geometric case, this follows from the results of \cite{hamilton_separability_2013} and \cite{wilton_profinite_2010}.)  By Proposition \ref{prop:normal}, a procyclic normal subgroup $\wh{Z}$ of $\wh{\pi_1M}$ is contained in a vertex stabilizer, and hence acts trivially on the profinite tree \cite[Theorem 2.12]{zalesskii_subgroups_1988}.  This would contradict the fact that the action is acylindrical unless $\wh{Z}=1$.
\end{proof}

A compact Seifert fibred 3-manifold $M$ can admit any of six different geometries.  We next explain how these geometric structures are also distinguished by the profinite completion of the fundamental group.  The reader is referred to \cite{scott_geometries_1983} for a details of the theory of Seifert fibred 3-manifolds. Consider the corresponding short exact sequence
\[
1\to Z\to \pi_1M\to \pi_1O\to 1
\]
where $Z$ is cyclic and $O$ is a compact, cone-type 2-orbifold.   The geometric structure of $M$ is determined by two invariants: the geometric structure of $O$; and the `Euler number' $e(M)$ \cite[Table 4.1]{scott_geometries_1983}.

We first note that the geometry of $O$ is detected by the profinite completion of $\pi_1M$.

\begin{lemma}\label{lem: orbifold geometry}
Let $M$ be a compact, Seifert fibred 3-manifold, as above:
\begin{enumerate}
\item $O$ is spherical if and only if $\wh{\pi_1M}$ is virtually procyclic;
\item $O$ is Euclidean if and only if $\wh{\pi_1M}$ is virtually nilpotent.
\end{enumerate}
\end{lemma}
\begin{proof}
This follows immediately from the following facts: $O$ is spherical if and only if $\pi_1M$ is virtually cyclic; $O$ is Euclidean if and only if $\pi_1M$ is virtually nilpotent; $\pi_1M$ is residually finite.
\end{proof}

If $M$ is not spherical then $O$ is a good orbifold, and so is covered by a compact, orientable surface $\Sigma$. The short exact sequence for $\pi_1M$ then pulls back along the covering map $\Sigma\to O$ to a short exact sequence
\[
1\to Z\to \pi_1N\to \pi_1\Sigma\to 1
\]
where $N$ is a finite-sheeted covering space of $M$.   We next explain how to the Euler number of $M$ manifests itself in the structure of $\pi_1M$.  The following lemma can be deduced from \cite{scott_geometries_1983}.

\begin{lemma}\label{lem: Euler number}
Let $M$ be a compact, Seifert fibred 3-manifold.
\begin{enumerate}
\item If $O$ is spherical then $e(M)= 0$ if and only $\pi_1M$ is infinite.
\item If $O$ is not spherical then $e(M)=0$ if and only if, for some finite-sheeted surface cover $\Sigma$ of $O$, the short exact sequence
\[
1\to Z\to\pi_1N\to\pi_1\Sigma\to 1
\]
splits.
\end{enumerate}
\end{lemma}

Finally, we explain how the splitting of the short exact sequence
\[
1\to Z\to\pi_1N\to\pi_1\Sigma\to 1
\]
is detected by the profinite completion.

\begin{lemma}\label{lem: Split sequences}
Let $\Sigma$ be a compact, orientable surface.  Then a short exact sequence
\begin{equation}\label{eqn: SF sequence}
1\to \Z\to G\to \pi_1\Sigma\to 1
\end{equation}
splits if and only if the induced exact sequence of profinite completions
\begin{equation}\label{eqn: Profinite SF sequence}
1\to \wh{\Z}\to \wh{G}\to \wh{\pi_1\Sigma}\to 1
\end{equation}
splits.
\end{lemma}
\begin{proof}
If (\ref{eqn: SF sequence}) splits then clearly so does (\ref{eqn: Profinite SF sequence}). Suppose therefore that (\ref{eqn: SF sequence}) does not split.

The exact sequence $1\to \Z\to \Z\ \to \Z/n\to 1$ induces a long exact sequence in cohomology, part of which is
\[
\cdots \to H^2(\pi_1\Sigma,\Z)\cong \Z\to H^2(\pi_1\Sigma,\Z/n)\to H^1(\pi_1\Sigma,\Z)\to\cdots
\]
But $H^2(\pi_1\Sigma,\Z/n)\cong\Z/n$, whereas $H^1(\pi_1\Sigma,\Z)$ is isomorphic to the abelianization of $\pi_1\Sigma$, which is torsion-free.  Therefore, the induced map $H^2(\pi_1\Sigma,\Z)\to H^2(\pi_1\Sigma,\Z/n)$ is surjective. In particular, if (\ref{eqn: SF sequence}) does not split then neither does the corresponding extension $1\to \Z/n\to G\to\pi_1\Sigma\to 1$, for some $n$.

Since $\pi_1\Sigma$ is good, we have that $H^2(\widehat{\pi_1\Sigma},\Z/n)\to H^2(\pi_1\Sigma,\Z/n)$, and so the corresponding sequence $1\to\Z/n\to\wh{G}\to\wh{\pi_1\Sigma}\to 1$ does not split.  Therefore, (\ref{eqn: Profinite SF sequence}) does not split, as required.
\end{proof}

We can summarize our results so far in the following theorem,
which asserts that the profinite completion detects the geometric
structure of a closed, orientable, irreducible 3-manifold.

\begin{theorem}\label{thm: Geometry}
Let $M,N$ be closed, orientable, irreducible 3-manifolds, and suppose that $\wh{\pi_1M}\cong\wh{\pi_1N}$.  Then $M$ admits one of Thurston's eight geometric structures if and only if $N$ does, and in this case both $M$ and $N$ admit the same geometric structure.
\end{theorem}
\begin{proof}
By Theorem \ref{thm: Profinite hyperbolization}, $M$ is hyperbolic if and only if $N$ is.  Since 3-manifold groups are residually finite and admit Sol geometry if and only if they are solvable but not virtually nilpotent, it follows that $M$ admits Sol geometry if and only if $N$ does.  Finally, by Theorem \ref{thm: Profinite Seifert conjecture}, $M$ is Seifert fibred if and only if $N$ is, and by Lemmas \ref{lem: orbifold geometry}, \ref{lem: Euler number} and \ref{lem: Split sequences}, they are of the same geometric type.
\end{proof}

\section{3-manifold groups}

In this section, we complete the proof of Theorem \ref{thm: 3M Tits alternative}.  We start by addressing the cusped hyperbolic case, and then deal with the Seifert fibred case, before concluding with the irreducible non-geometric and reducible cases.

\subsection{The cusped case}\label{subsec: Cusp}

We need to generalize the results of the previous section to the case of non-compact hyperbolic 3-manifolds of finite volume.  As in Corollary 14.16 of \cite{wise_structure_2012}, the strategy is to cut along a family of surfaces, in such a way that the pieces are hyperbolic without cusps.  We next exhibit the family of surfaces that we will need (cf.\ \cite[Corollary 14.16]{wise_structure_2012}).

\begin{theorem}\label{thm: Cusped hierarchy}
A hyperbolic 3-manifold $M$ with cusps has a finite-sheeted covering space $N\to M$ that contains a disjoint family of connected, geometrically finite, incompressible subsurfaces $\{\Sigma_1,\ldots,\Sigma_n\}$ such that:
\begin{enumerate}
\item each cusp of $N$ contains a boundary component of some $\Sigma_i$;
\item each $\pi_1\Sigma_i\subseteq \pi_1 N$ is relatively malnormal.
\end{enumerate}
\end{theorem}
\begin{proof}
By the `Half Lives, Half Dies' Lemma, there is a homomorphism $\pi_1 M\to\mathbb{Z}$ such that each peripheral subgroup maps non-trivially.  One can realize this by a smooth map $M\to S^1$, and the preimage of a generic point is a (possibly disconnected) surface that cuts every cusp, which can be compressed to obtain an incompressible surface with the same property.  Passing to a finite-sheeted cover and applying an argument with the Thurston norm, one can further ensure that each component is geometrically finite.

Let $\Sigma$ be such a component.  Geometrical finiteness implies that there are finitely many double cosets $\pi_1\Sigma g_i\pi_1\Sigma$ such that the intersection $\pi_1\Sigma^{g_i}\cap \pi_1\Sigma$ is non-peripheral. By subgroup separability, we can pass to a finite-sheeted covering space to which $\Sigma_i$ lifts but the elements $g_i$ do not.  Doing this for every component, and passing to a deeper regular covering space (still of finite index), we obtain the required covering space.
\end{proof}

Wise used a similar result to prove that cusped hyperbolic manifolds are virtually compact special \cite[Theorem 14.29]{wise_structure_2012}.  We state this result here for future use.

\begin{theorem}[Wise]\label{thm: Wise cusped}
Let $M$ be a cusped hyperbolic 3-manifold of finite volume.  Then $\pi_1M$ is virtually compact special.
\end{theorem}

Cutting along the family of surfaces given by Theorem \ref{thm: Cusped hierarchy} produces a graph-of-groups decomposition $(\G,\Delta)$ for $\pi_1N$, with the property that every vertex and edge group is hyperbolic and virtually special, and the stabilizer of any infinite subtree of the Bass--Serre tree is cyclic.  Note also that $(\G,\Delta)$ is efficient.  Passing to the corresponding graph of profinite groups $(\bar{\G},\Delta)$, it follows from Theorem \ref{thm: Profinite relative malnormality} that the stabilizer of any infinite subtree of the standard profinite tree is pro-cyclic (possibly trivial).

We next state our main theorem in the cusped case.

\begin{theorem}\label{thm: Cusped case}
Let $M$ be a cusped hyperbolic 3-manifold, $\pi_1M$ its fundamental group and $\widehat{\pi_1M}$ its profinite completion. If $H$ is a closed subgroup of $\wh{\pi_1M}$ that does not contain a free non-abelian pro-$p$ subgroup for any $p$ then $H$ is  isomorphic to $\Z_\pi\rtimes\Z_\rho$. Furthermore, if $H$ is not projective (i.e. $\pi\cap \rho\neq \emptyset$)  then $H$ is conjugate into the closure of a cusp subgroup of $\pi_1M$.
\end{theorem}

\begin{proof} Put $H_0=\widehat{\pi_1N}\cap H$, where
$N$ is the cover from  Theorem \ref{thm: Cusped hierarchy}. We first prove that $H_0$ has the claimed semi-direct product structure.  Consider the action of $H_0$ on the standard profinite tree.  By Theorem \ref{thm: Hyp special} the stabilizers of vertices in $H_0$ have the claimed structure.  Since the stabilizers of infinite subtrees are pro-cyclic, the kernel $K$ of the action is torsion-free pro-cyclic.

Suppose first that $K$ is non-trivial.  The closure of any peripheral subgroup $\widehat P_i$ is malnormal in $\widehat{\pi_1M}$ by {Lemma \ref{lem: Peripheral profinite malnormality}}, and since up to conjugation $K$ is in $\widehat P_i$, we deduce that $K$ is central in $H$ and $H$ is abelian. So $H$ is a subgroup of $\widehat \Z \times \widehat \Z$ in this case.

Alternatively, if $K$ is trivial then $H_0$ is isomorphic to a soluble projective group $\Z_\pi\rtimes\Z_\rho$ by Theorem \ref{profinite}.  Observe now that $\widehat{\pi_1M}$ is torsion free, since $\Gamma$ is good (see Definition \ref{good} and \cite[Theorem 1.4]{grunewald_cohomological_2008} combined with Theorem  \ref{thm: Wise cusped}). So $H$ is torsion free and hence, by Remark \ref{cases}, a semi-direct product of the claimed form.

Finally, if $H$ is not projective then $K$ is necessarily non-trivial  by Theorem \ref{profinite}, and so the above argument shows that $H$ is conjugate into the closure of some $\widehat P_i$ as claimed.
\end{proof}

\subsection{Seifert fibred 3-manifolds}

In this section we prove Theorem \ref{thm: 3M Tits alternative}
when $M$ is Seifert fibred.  For many of the cases (if $M$ has
Euclidean, spherical, Nil or $S^2\times\R$ geometry), $\pi_1M$ is
virtually soluble, and there is nothing to prove.  The next
proposition addresses the remaining
geometries---$\mathbb{H}^2\times\R$ and $\widetilde{SL_2\R}$---in
which the base orbifold is hyperbolic.

\begin{proposition}\label{prop: SF Tits alternative}
Suppose that $M$ is any compact, Seifert fibred 3-manifold and
that $\pi_1M$ is not virtually soluble.  If $H$ is a closed
subgroup of $\wh{\pi_1M}$ that does not contain a free non-abelian
pro-$p$ subgroup for any $p$ then $H$  is isomorphic to
$(\Z_\sigma\times \Z_\pi)\rtimes\Z_\rho$, where
$\pi\cap\rho=\varnothing$ and $\Z_\rho$ acts on $\Z_\sigma$ by
inversion.
\end{proposition}

\begin{proof}
The fundamental group $G=\pi_1M$ fits into a short exact sequence
\[
1\to Z\to G\stackrel{p}{\to}\pi_1(O)\to 1
\]
where $Z=\langle z\rangle$ is cyclic and $O$ is a cone-type 2-orbifold (see, for instance, \cite{scott_geometries_1983}).

Since $G$ is not virtually soluble, $O$ is a hyperbolic orbifold,
which is finitely covered by an orientable hyperbolic surface
$\Sigma$; passing to a further cover if necessary, we may further
assume that $\Sigma$ admits an essential simple-closed $\gamma$
which is not boundary parallel and that $\pi_1\Sigma$ is
torsion-free.

The exact sequence above induces an exact sequence
\[
1\to \wh{Z}\to \wh{G}\stackrel{\hat{p}}{\to}\wh{\pi_1(O)}\to 1
\]
of profinite completions.  Then $\widehat{\pi_1\Sigma}$ is naturally
an open subgroup of $\widehat{\pi_1(O)}$, which pulls back to
an open subgroup $\widehat{\pi_1 N}$ of $\wh{G}$. Cutting $\Sigma$
along $\gamma$ induces a splitting of $\pi_1\Sigma$ in which the
edge group is malnormal; this in turn induces a graph of profinite
groups for $\widehat{\pi_1\Sigma}$ with procyclic malnormal edge group.

Since $\widehat{\pi_1\Sigma}$ is
torsion-free, by Theorem \ref{profinite} $\hat{p}(H)\cong \Z_{\bar
\pi}\rtimes \Z_{\bar\rho}$ is a projective soluble group with
$\bar\pi\cap \bar \rho=\varnothing$. Denote by $\delta$ the set of
primes dividing the orders of torsion elements of
$\widehat{\pi_1(O)}$. Then $\widehat Z=\Z_{\sigma}\times
\Z_{\delta}$ and $H/\Z_{\sigma}$ is a torsion-free soluble
group containing a projective group $U$. Hence  $H/\Z_{\sigma}$ is soluble projective (projective groups are groups of cohomological dimension 1 and torsion-free overgroups of finite
index preserve cohomological dimension---cf.\
\cite[Theorem 7.3.7]{ribes_profinite_2010}) and hence isomorphic to $\Z_\pi\rtimes \Z_\rho$ for
some  $\pi\cap \rho=\varnothing$ (see Exercise 7.7.8 in
\cite{ribes_profinite_2010}). Thus $H\cong \Z_{\sigma}\rtimes H/\Z_{\sigma}\cong \Z_{\sigma}\rtimes (\Z_\pi\rtimes \Z_\rho)$.
Since the action on $\widehat Z$ is induced from the action on $Z$
which is either trivial or by inversion, moving the prime $2$ from $\pi$ to
$\rho$ if necessary we can rewrite $H$ as $H\cong (\Z_{\sigma}\times \Z_\pi)\rtimes \Z_\rho$.
\end{proof}

\subsection{Irreducible 3-manifolds}

We next consider the general case in which $M$ is an arbitrary closed, orientable, irreducible 3-manifold.  The only remaining case to consider is when $M$ has non-trivial torus decomposition. Our previous results give an 4-acylindrical action on a profinite tree, and the argument will now go through as before.

\begin{proposition}\label{prop: Non-trivial JSJ alternative}
Let $M$ be a compact, orientable, irreducible 3-manifold whose
torus decomposition is non-trivial. If $H$ is a closed subgroup of
$\wh{\pi_1M}$ that does not contain a free non-abelian pro-$p$
subgroup for any $p$ then $H$ is isomorphic to $(\Z_\sigma\times
\Z_\pi)\rtimes\Z_\rho$, where $\Z_\rho$ acts on $\Z_\sigma$ by
inversion and $\pi\cap \rho=\emptyset$.
\end{proposition}

\begin{proof}
By the results of \cite{hamilton_separability_2013} and \cite{wilton_profinite_2010}, $\widehat{\pi_1M}$ acts 4-acylindrically on a profinite tree.  Therefore so does $H$, and by Proposition \ref{prop: SF Tits alternative} and  Theorem \ref{thm: Cusped case} the stabilizers of vertices in $H$ have the claimed structure. So if $H$ stabilizes a vertex we are done. Otherwise,  since the action is 4-acylindrical the kernel $K$ is trivial. Then $H$ is isomorphic to a soluble projective group $\Z_\pi\rtimes\Z_\rho$ by Remark \ref{cases} since $\widehat{\pi_1M}$ is torsion free.
\end{proof}

Since we have already dealt with the Seifert-fibred and hyperbolic
cases, and the fundamental group of a Sol-manifold is soluble,
Theorem \ref{thm: 3M Tits alternative} follows whenever $M$ is
closed, orientable and irreducible.

\subsection{Reducible 3-manifolds}\label{Reducible 3-manifolds}

To finish the proof of Theorem \ref{thm: 3M Tits alternative}, we double to obtain a retract of a closed 3-manifold $M$. Now we argue as before, using the Grushko decomposition.  After doubling, we may assume that $M$ is closed.  We then have
\[
\pi_1M\cong A_1*\ldots*A_m*B_1*\ldots*B_n*F
\]
where the $A_i$ are all finite, the $B_i$ are all as considered in the previous section, and $F$ is free.

\begin{proof}[Proof of Theorem \ref{thm: 3M Tits alternative}]
We consider the action of
\[
\widehat{\pi_1M}= \wh{A}_1\amalg\ldots\amalg \wh{A}_m\amalg \wh{B}_1\amalg \ldots \amalg \wh{B}_n\amalg \wh{F}
\]
on the standard profinite tree $S$ associated to the free
profinite product decomposition of $\widehat{\pi_1M}$. If $H$
stabilizes a vertex then $H$ is conjugate to a subgroup of a free
factor, so the structure of $H$ is described in the previous
results. Otherwise, it follows from Theorem \ref{profinite}
combined with Remark \ref{cases} that $H$ is as in  item 2  of
Theorem \ref{thm: 3M Tits alternative}.\end{proof}

\section{One-relator groups with torsion}

One-relator groups with torsion are hyperbolic, and Wise proved further that they are virtually special.  Theorem \ref{thm: Hyp special} therefore applies.  However, by carefully examining the hierarchy that Wise used in his proof, we can improve the conclusions and prove Theorem \ref{thm: 1-relator}.

\begin{proof}[Proof of Theorem \ref{thm: 1-relator}]
Every one-relator group $G$ embeds naturally into a free product $G'=G\ast \Z$ which is an HNN extension $HNN(L,M,t)$ of a simpler one-relator group $H$, where $M$ is a free subgroup generated by
subsets of the generators of the presentation of $G$ (cf.\ the
Magnus--Moldavanskii construction in Section~18b \cite{wise_structure_2012}
). The hierarchy terminates at a virtually free group of the form
$\Z/n\ast F$, where $F$ is free. We use induction on the length of this hierarchy.

If $G$ has torsion, this hierarchy is quasiconvex (see \cite[Lemma
18.8]{wise_structure_2012}) and so (at each step of the hierarchy) the
edge group $K$ has finite width by Theorem \ref{thm: GMRS}.  By
Corollary \ref{cor: Finite width closure},  $\widehat{K}$ has
finite width in $\widehat{G}'$. It follows that the stabilizer in
$\widehat{G}'$ of any infinite subtree of the corresponding
profinite standard tree $S$ is trivial.

Consider the action of $H$ on its minimal invariant profinite subtree. If $H$ stabilizes a vertex, then by induction on the length of the hierarchy we are done.  Otherwise, by Theorem \ref{profinite}, $H$ is of one of the claimed forms.  This finishes the proof.
\end{proof}

\section{Pro-$p$ subgroups}

Finally, we study the pro-$p$ subgroups of profinite groups acting acylindrically on profinite trees.  The following general theorem is the main result of this section.

\begin{theorem}\label{pro-p acting on profinite}
Let $G$ be a finitely generated pro-$p$ subgroup of a profinite group $\Gamma$ acting 1-acylindrically on a profinite tree $T$.  Then $G$ is a free pro-$p$ product of vertex stabilizers and a free pro-$p$ group.\end{theorem}

We start with the following.

\begin{lemma}\label{prosoluble} Let $G$ be a profinite group acting on a profinite
tree $T$  and let $e$ be an edge of $T$. Suppose the stabilizers of
the vertices $v,w$ of the edge $e$ are prosoluble and do not
coincide with $G_e$.  Then the group $H=\langle G_v,G_w\rangle$
has a free prosoluble amalgamated product $G_v\amalg_{G_e} G_w$ as
a quotient. In particular, $H$ is not pro-$p$.\end{lemma}

\begin{proof}
Let $D=H(e\cup v\cup w)$. Then for any open subgroup $U$ of $H$
the quotient graph $U\backslash D$ is a finite connected quotient
graph of $D$ and since
$D=\lim\limits_{\displaystyle\longleftarrow} U\backslash D$, $D$
is a connected profinite subgraph of $T$.  Then $D$ is a profinite
tree \cite[(1.15)]{zalesskii_subgroups_1988}.  Since $H$ is
generated by the vertex stabilizers, $H\backslash D$ is a
profinite tree \cite[Proposition 2.5]{zalesskii_subgroups_1988}
and so $\{v,w\}$ maps to $H\backslash D$ injectively, i.e.\
$H\backslash D$ is isomorphic to $v\cup e \cup w$. Let
$f:H\longrightarrow H_s$ be the maximal prosoluble quotient and
let $K$ be the kernel of $f$. Then
 $K\backslash D$ is a profinite tree \cite[Proposition 2.9(b)]{zalesskii_geometric_1989} and in particular is a prosoluble simply connected graph. Therefore by the prosoluble version of
\cite[Proposition 4.4]{zalesskii_fundamental_1989} (it is remarked
in 5.4 there that it is valid in the prosoluble case) $H/K$ is
isomorphic   to the prosoluble fundamental group
$\Pi_1^s(\curlyH,H\backslash D)$ of the graph of groups $(\curlyH,H\backslash
D)$ where edge and vertex groups are the stabilizers of $e$ and
vertices $v,w$. This shows that $H/K$ is a free prosoluble
amalgamated product $G_v\amalg_{G_e}^s G_w$.
\end{proof}

We are now ready to prove the main theorem of this section.

\begin{proof}[Proof of Theorem \ref{pro-p acting on profinite}] First note that a nontrivial  stabilizer of any edge
$e$ coincides with one of its vertex stabilizers $G_v$ or $G_w$,
say $G_w$,  since otherwise by Lemma \ref{prosoluble} the
stabilizers of two vertices of  $e$ do not generate a pro-$p$
group. Moreover, if $e'$ is another edge with vertices $w$
and $u$ then $G_{e'}=\{1\}$ since
because of 1-acylindricity $G_w=G_e\cap G_e'=1$. It follows that
the connected components of the abstract subgraphs of points with
non-trivial stabilizers are at most stars. Let $\widetilde G$ be
the subgroup of $G$ generated by all vertex stabilizers of $G$.
Then by \cite[Theorem 2.6]{zalesskii_subgroups_1988} $G/\widetilde
G$ is a free pro-$p$ group and we denote its retract in $G$ by
$F_0$.

Now since $G$ is finitely generated, its Frattini series
$\Phi^n(G)$ is a fundamental system of neighbourhoods of 1. Let
$\widetilde{\Phi^n(G)}$ be  the subgroup of $\Phi^n(G)$ generated
be all vertex stabilizers of $\Phi^n(G)$. Then
$G_n=G/\widetilde{\Phi^n(G)}$  acts on a profinite tree
$T_n=\widetilde{\Phi^n(G)}\backslash T$ \cite[Proposition
2.5]{zalesskii_subgroups_1988} and
$\Phi^n(G)/\widetilde{\Phi^n(G)}$ acts freely on $T_n$ and
therefore is free pro-$p$ \cite[Theorem
2.6]{zalesskii_subgroups_1988}. Note that the vertex and edge
stabilizers of $G_n$ acting on $T_n$ are finite epimorphic images
of the corresponding vertex and edge stabilizers of $G$ acting on
$T$ and so the images in $T_n$ of edges of $T$ with trivial edge
stabilizers  have  trivial stabilizers. Therefore the (abstract)
connected components of the subgraph of points of $T_n$ with
non-trivial stabilizers in $G_n$  are still at most stars. Indeed,
let $S_n$ be the image of a star $S$ of $T$ with non-trivial edge
stabilizers and let $e_n$ be an edge not in $S_n$ having a vertex
$v_n\in S_n$ with non-trivial stabilizer. Let $v$ be a vertex of
$S$ whose image in $T_n$ is $v_n$. Then $G_v\neq \{1\}$ and there
exists an edge $e$ incident to $v$ whose image in $T_n$ is $e_n$.
But $G_e$ is trivial and so $G_{e_n}$ is trivial.

 By
\cite[Lemma 8]{herfort_virtually_2008} a virtually free pro-$p$
group has only finitely many finite subgroups up to conjugation
and so $T_n$ has only finitely many edges with non-trivial edge
stabilizers up to translation. Therefore the subgraph  $\Sigma_n$
of points with non-trivial edge stabilizers is closed in $T_n$,
i.e. is a profinite subgraph (forest) of $T_n$.  Collapsing all
connected components (stars) of $\Sigma_n$ in $T_n$, by the
Proposition on page 486 in \cite{zalesskii_open_1992} we get a pro-$p$ tree
$T_n$ on which $G_n$ acts with trivial edge stabilizers, so by
\cite[Proposition 2.4]{herfort_splitting_2013} $G_n$ is a free
pro-$p$ product
$$G_n=G/\widetilde{\Phi^n(G)}=\left(\coprod_{v} G_{nv}\right)\amalg F_{0n},$$
of representatives (chosen arbitrarily)  of the non-trivial vertex
stabilizers and of the isomorphic image $F_{0n}$ of $F_0$ in
$G_n$.

 Since $G=\lim\limits_{\displaystyle\longleftarrow} G_n$ and $G$ is
finitely generated, by choosing $n$ large enough we may assume that
the number of free factors is the same for every $m>n$, i.e.\ $v$
ranges over a finite set $V$. By
\cite[Theorem 4.2]{ribes_pro-p_2000} every finite subgroup of a free pro-$p$ product is
conjugate to a subgroup of a free factor. Therefore, the free
factors of $G_{m+1}$ are mapped onto the free factors of $G_m$ up
to conjugation. But in a free pro-$p$ product decomposition
replacing  any free factor by its conjugate does not change the
group. So, starting from $n$, we can inductively choose
$G_{m+1v}$ in such a way that its image in $G_m$ is $G_{mv}$. Let
$G_v$ be the inverse limit of $G_{mv}$. Then by
\cite[Lemma 9.1.5]{ribes_profinite_2010} $G= \coprod_{v\in V} G_v \amalg F_0$.

It remains to observe that $G_{mv}$ is a stabilizer of a vertex in
$T_m$ so the set of fixed points $T^{G_{mv}}$ is non-empty and
closed \cite[Theorem 3.7]{ribes_pro-p_2000}. Therefore, $G_v$
is the stabilizer of a non-empty set of vertices
$\lim\limits_{\displaystyle\longleftarrow} T^{G_{vm}} $. This
finishes the proof.  \end{proof}

Using  Theorem \ref{pro-p acting on profinite} we can precisely describe the finitely generated pro-$p$ subgroups of the profinite completions of torsion-free hyperbolic virtually special groups.

\begin{proof}[Proof of Theorem \ref{thm: pro-p virtually special}]
By Theorem \ref{thm: Hierarchies}, Theorem \ref{thm: Malnormal
closure} and Lemma \ref{lem: Acylindrical action}, $G$ has a
subgroup of finite index $\Gamma_0$ whose profinite completion
acts 1-acylindrically on a profinite tree. Then by Theorem
\ref{pro-p acting on profinite} $H\cap \widehat \Gamma_0$ is a
free pro-$p$ product of vertex stabilizers and a free pro-$p$
group. By induction on the hierarchy vertex stabilizers are free
pro-$p$, therefore so is $H\cap \widehat \Gamma_0$. Finally by
Serre's theorem (see Theorem 7.3.7 in \cite{ribes_profinite_2010}) $H$ is free
pro-$p$.
\end{proof}

Using the main result of \cite{herfort_virtually_2013} that describes finitely
generated virtually free pro-$p$ groups, we can deduce a
description of the finitely generated pro-$p$ subgroups of the
profinite completion of a hyperbolic virtually special group.

\begin{corollary}\label{cor:pro-p virtually special} Let $G$ be a  hyperbolic, virtually special
group.  Any finitely generated pro-$p$ subgroup $H$ of $\widehat
G$ is the fundamental pro-$p$ group of a finite graph of finite
$p$-groups.\end{corollary}

\begin{proof} The group $G$  contains a virtually special, torsion-free subgroup of finite
index $U$. Hence $H\cap \widehat U$ is free pro-$p$. The result
then follows from \cite[Theorem 1.1]{herfort_virtually_2013}.\end{proof}

Of course, Corollary \ref{cor:pro-p virtually special} applies to one-relator groups with torsion by the work of Wise \cite{wise_structure_2012}. However, a careful analysis gives a more refined classification in that case.

\begin{theorem}
Let $G$ be a one-relator group with torsion and $H$ a
finitely generated pro-$p$ subgroup of $\wh{G}$. Then $G$ is a
free pro-$p$ product of finite cyclic $p$-groups and a free
pro-$p$ group.
\end{theorem}

\begin{proof} The group $G$  contains a virtually special torsion-free subgroup of finite
index $U$. Hence every finitely generated pro-$p$ subgroup of
$\widehat U$ is free pro-$p$.

Now we show that centralizers of torsion elements of $\widehat G$
are finite. First observe that   all torsion  elements of
$\widehat G$ are conjugate to elements of $G$ (see Theorem 2.1
\cite{boggi_restricted_2014}). But by \cite[Theorem 1.1]{minasyan_one-relator_2013} combined with \cite[Proposition 3.2]{minasyan_hereditary_2012} the centralizer of any element of $G$ is dense in the centralizer of this element in $\widehat G$.
Thus since the centralizer of any torsion element of $G$ is finite
\cite[Theorem 2]{newman_some_1968}, the centralizers of torsion elements
in $\widehat G$ are finite as well.

 So any   finitely generated pro-$p$ subgroup of $\widehat G$  is virtually free and has finite centralizers of torsion elements.  Therefore by \cite[Theorem 1]{herfort_virtually_2008} it has the claimed structure.
 \end{proof}

\bibliographystyle{alpha}



\end{document}